\documentclass[a4paper,11pt]{article}
\usepackage{makeidx}
\usepackage{graphicx}
\usepackage[english,activeacute]{babel}
\usepackage[latin1]{inputenc}
\usepackage{amsmath}
\usepackage{amsfonts}
\usepackage{amssymb,latexsym}
\usepackage{colortbl}

\newtheorem{theorem}{Theorem}

\newtheorem{corollary}{Corollary}

\newtheorem{lemma}{Lemma}

\newtheorem{proposition}{Proposition}
\newtheorem{remark}{Remark}

\newenvironment{proof}[1][Proof]{\noindent\textbf{#1.} }{\ \rule{0.5em}{0.5em}}
\begin{document}

\title{Zeros of orthogonal polynomials generated by the Geronimus perturbation of measures}
\author{Am\'ilcar Branquinho$^{1}$, Edmundo J. Huertas$^{2,\dag}$, Fernando
R. Rafaeli$^{3}$\thanks{
} \\
$^{1,2}$CMUC, Departamento de Matemática (FCTUC)\\
Universidade de Coimbra, Portugal\\
ajplb@mat.uc.pt, ehuertas@mat.uc.pt\\
$^{3}$Faculdade de Matem\'atica,\\
Universidade Federal de Uberl\^andia (UFU), 38.408-100 Uberl\^andia, Brazil\\
fernando.rodrigo.rafaeli@gmail.com}
\date{\emph{(\today)}}
\maketitle

\begin{abstract}
This paper deals with monic orthogonal polynomial sequences (MOPS in short)
generated by a Geronimus canonical spectral transformation of a positive
Borel measure $\mu $, i.e., 
\begin{equation*}
\frac{1}{(x-c)}d\mu (x)+N\delta (x-c),
\end{equation*}%
for some free parameter $N\in \mathbb{R}_{+}$\ and shift $c$. We analyze the
behavior of the corresponding MOPS. In particular, we obtain such a behavior
when the mass $N$ tends to infinity as well as we characterize the precise
values of $N$ such the smallest (respectively, the largest) zero of these
MOPS is located outside the support of the original measure ~$\mu $. When $%
\mu $ is semi-classical, we obtain the ladder operators and the second order
linear differential equation satisfied by the Geronimus perturbed MOPS, and
we also give an electrostatic interpretation of the zero distribution in
terms of a logarithmic potential interaction under the action of an external
field. We analyze such an equilibrium problem when the mass point of the
perturbation $c$ is located outside the support of $\mu $.

\medskip


\medskip


\medskip

\dag Corresponding author.
\end{abstract}


\section{Introduction}

\label{[SECTION-1]-Intro}



\subsection{Geronimus perturbation of a measure}


In the last years some attention has been paid to the so called canonical
spectral transformations of measures. Some authors have analyzed them from
the point of view of Stieltjes functions associated with such a kind of
perturbations (see \cite{Z-JCAM97}) or from the relation between the
corresponding Jacobi matrices (see \cite{Y-BKMS02}). The present
contribution is focused on the behavior of zeros of monic orthogonal
polynomial sequences (MOPS in the sequel) associated with a particular
transformation of measures called the \textit{Geronimus canonical
transformation on the real line}. Let $\mu $ be an absolutely continuous
measure with respect to the Lebesgue measure supported on a finite or
infinite interval $E= \mathrm{supp}(\mu )$, such that $C_{0}(E)=[a,b]%
\subseteq \mathbb{R}$. The basic Geronimus perturbation of $\mu $ is defined
as%
\begin{equation}
d\nu _{N}(x)=\frac{1}{(x-c)}d\mu (x)+N\delta (x-c),  \label{[S1]-GeronimusTr}
\end{equation}%
with $N\in \mathbb{R}_{+}$, $\delta (x-c)$ the Dirac delta function in $x=c$%
, and the shift of the perturbation verifies $c\not\in E$. Observe that it
is given simultaneously by a rational modification of $\mu $ by a positive
linear polynomial whose real zero $c$ is the point of transformation (also
known as the \textit{shift} of the transformation) and the addition of a
Dirac mass at the point of transformation as well.

This transformation was introduced by Geronimus in the seminal papers \cite%
{G-HNeds40} and \cite{G-ZMOK40} devoted to provide a procedure of
constructing new families of orthogonal polynomials from other orthogonal
families, and also was studied by Shohat (see \cite{S-TAMS37}) concerning
about mechanical quadratures. The problem was revisited by Maroni in \cite%
{M-PMH90}, into a more general algebraic frame, who gives an expression for
the MOPS associated with (\ref{[S1]-GeronimusTr}) in terms of the so called
co-recursive polynomials of the classical orthogonal polynomials. In the
past decade, Bueno and Marcellán reinterpreted this perturbation in the
framework of the so called discrete Darboux transformations, $LU$ and $UL$
factorizations of shifted Jacobi matrices \cite{BM-LAA04}. This
interpretation as Darboux transformations, together with other canonical
transformations (Christoffel and Uvarov), provide a link between orthogonal
polynomials and discrete integrable systems (see \cite{BB-JDEA09}, \cite%
{SZ-MAA95} and \cite{SZ-JPA97}). More recently, in \cite{BDT-NA10} the
authors present a new computational algorithm for computing the Geronimus
transformation with large shifts, and \cite{DM-NA13} concerns about a new
revision of the Geronimus transformation in terms of symmetric bilinear
forms in order to include certain Sobolev and Sobolev--type orthogonal
polynomials into the scheme of Darboux transformations.

In order to justify the relevance of this contribution, we point out that
the behavior of the zeros of orthogonal polynomials is extensively studied
because of their applications in many areas of mathematics, physics and
engineering. Following this premise, the purpose of this paper is twofold.
First, using a similar approach as\ was done in \cite{HMR-AMC12}, we provide
a new connection formula for the Geronimus perturbed MOPS, which will be
crucial to obtain sharp limits (and the speed of convergence to them) of
their zeros. We provide a comprehensive study of the zeros in terms of the
free parameter of the perturbation $N$, which somehow determines how
important the perturbation on the classical measure $\mu $ is. Notice that
this work also concerns with the behavior of the eigenvalues of the monic
Jacobi matrices associated to certain Darboux transformations with shift $c$
and free parameter $N$ studied in \cite{BDT-NA10}. Second, from the
aforementioned new connection formula we recover (from an alternative point
of view) a connection formula already known in the literature (see \cite%
{M-PMH90}) in terms of two consecutive polynomials of the original measure $%
\mu $. We also obtain explicit expressions for the ladder operators and the
second order differential equation satisfied by the Geronimus perturbed
MOPS. When the measure $\mu $ is semi-classical, we also obtain the
corresponding electrostatic model for the zeros of the Geronimus perturbed
MOPS, showing that they are the electrostatic equilibrium points of positive
unit charges interacting according to a logarithmic potential under the
action of an external field (see, for example, Szeg\H{o}'s book \cite[%
Section 6.7]{Sze75}, Ismail's book \cite[Ch. 3]{Ism05} and the references
therein).

The structure of the paper is as follows. The rest of this Section is
devoted to introduce without proofs some relevant material about modified
inner products and their corresponding MOPS. In Section \ref%
{[SECTION-2]-MainRes} we provide our main results. We obtain a new
connection formula for orthogonal polynomials generated by a basic Geronimus
transformation of a positive Borel measure $\mu $, sharp bounds and speed of
convergence to them for their real zeros, and the ladder operators and the
second linear differential equation that they satisfy. The results about the
zeros follows from a lemma concerning the behavior of the zeros of a linear
combination of two polynomials. In Section \ref{[SECTION-3]-Proofs}, we
proof all the result provided in the former Section. Finally, in Section \ref%
{[SECTION-4]-Examples}, we explore these results for the Geronimus perturbed
Laguerre and Jacobi MOPS. For $\mu $ being semi-classical, we obtain the
corresponding electrostatic model for the zeros of the Geronimus perturbed
MOPS as equilibrium points in a logarithmic potential interaction of
positive unit charges under the presence of an external field. We analyze
such an equilibrium problem when the mass point is located outside the
support of the measure ~$\mu$, and we provide explicit formulas for the
Laguerre and Jacobi weight cases.


\subsection{Modified inner products and notation}


Let $\mu $\ be a positive Borel measure $\mu $, with existing moments of all
orders, and supported on a subset $E\subseteq \mathbb{R}$ with infinitely
many points. Given such a measure, we define the standard inner product $%
\langle \cdot ,\cdot \rangle _{\mu }:\mathbb{P}\times \mathbb{P}\rightarrow 
\mathbb{R}$ by%
\begin{equation}
\langle f,g\rangle _{\mu }=\int_{E}f(x)g(x)d\mu (x),\quad f,g\in \mathbb{P},
\label{[S1]-InnProd-1}
\end{equation}%
where $\mathbb{P}$\ is the linear space of the polynomials with real
coefficients, and the corresponding norm $||\cdot ||_{\mu }:\mathbb{P}%
\rightarrow \lbrack 0,+\infty )$ is given, as usual, by%
\begin{equation*}
||f||_{\mu }=\sqrt{\int_{E}|f(x)|^{2}d\mu (x)},\quad f\in \mathbb{P}.
\end{equation*}%
Let $\{{P_{n}\}}_{n\geq 0}$ be the MOPS associated with $\mu $. It is very
well known that the former MOPS satisfy the three term recurrence relation%
\begin{eqnarray}
xP_{n}(x) &=&P_{n+1}(x)+\beta _{n}P_{n}(x)+\gamma _{n}P_{n-1}(x),
\label{[S1]-3TRR-Pn} \\
P_{-1}(x) &=&0,\quad P_{0}(x)=1.  \notag
\end{eqnarray}%
If%
\begin{equation}
K_{n}(x,y)=\sum\limits_{k=0}^{n}\frac{P_{k}(x)P_{k}(y)}{||P_{k}||_{\mu }^{2}}
\label{[S1]-Kernel-n}
\end{equation}%
denotes the corresponding $n$-th kernel polynomial, according to the
Christoffel-Darboux formula, for every $n\in \mathbb{N}$ we have%
\begin{equation*}
K_{n}(x,y)=\frac{P_{n+1}(x)P_{n}(y)-P_{n+1}(y)P_{n}(x)}{(x-y)}\frac{1}{%
||P_{n}||_{\mu }^{2}}.
\end{equation*}%
Notice that this structures satisfy the well-known \textquotedblleft
reproducing property\textquotedblright\ of the $n$-th kernel polynomial%
\begin{equation*}
\int_{E}K_{n}\left( x,y\right) f\left( x\right) d\mu (x)=f\left( y\right)
\end{equation*}%
for any polynomial $f\in \mathbb{P}$\ with $\deg \,(f)\leq n$.

Here and subsequently, $\{{P_{n}^{c,[k]}\}}_{n\geq 0}$ denotes the MOPS with
respect to the modified inner product%
\begin{equation}
\langle f,g\rangle _{\mu ,[k]}=\int_{E}f(x)g(x)(x-c)^{k}d\mu (x),
\label{[S1]-InnProd-2}
\end{equation}%
where $c\notin E= \mathrm{supp}(\mu )$. The polynomials $\{{P_{n}^{c,[k]}\}}%
_{n\geq 0}$ are orthogonal with respect to a polynomial modification of the
measure $\mu $ called the $k$\textit{-iterated Christoffel perturbation}. If 
$k=1$ we have the \textit{Christoffel canonical transformation of the
measure }$\mu $\textit{\ }(see \cite{Z-JCAM97} and \cite{Y-BKMS02}). It is
well known that $P_{n}^{c,[1]}(x)$ is the monic kernel polynomial which can
be represented as (see \cite[(7.3)]{Chi78})%
\begin{equation}
P_{n}^{c,[1]}(x)=\frac{1}{(x-c)}\left( P_{n+1}(x)-\pi _{n}\,P_{n}(x)\right) =%
\frac{\Vert P_{n}\Vert _{\mu }^{2}}{P_{n}(c)}K_{n}(x,c),  \label{[S1]-Pc1n}
\end{equation}%
with%
\begin{equation}
\pi _{n}=\pi _{n}(c)=\frac{P_{n+1}(c)}{P_{n}(c)}.  \label{[S1]-param-pi}
\end{equation}%
Notice that $P_{n}^{c,[1]}(c)\neq 0$. We will denote%
\begin{equation*}
||P_{n}^{c,[k]}||_{\mu ,[k]}^{2}=\int_{E}|P_{n}^{c,[k]}(x)|^{2}(x-c)^{k}d\mu
.
\end{equation*}

Next, let us consider the basic Geronimus perturbation of $\mu $ given in (%
\ref{[S1]-GeronimusTr}). Let $\{Q_{n}^{c}\}_{n\geq 0}$\ be the MOPS
associated with $d\nu _{N}(x)$ when the $N=0$. That is, they are orthogonal
with respect to the measure%
\begin{equation}
d\nu _{N=0}(x)=d\nu (x)=\frac{1}{(x-c)}d\mu (x).  \label{[S1]-LinearDiv}
\end{equation}%
This constitutes a linear rational modification of $\mu $, and the
corresponding MOPS $\{Q_{n}^{c}\}_{n\geq 0}$ with respect to%
\begin{equation}
\langle f,g\rangle _{\nu }=\int_{E}f(x)g(x)d\nu (x)=\int_{E}f(x)g(x)\frac{1}{%
(x-c)}d\mu (x)  \label{[S1]-InnProd-3}
\end{equation}%
has been extensively studied in the literature (see, among others, \cite%
{BM-IJMMS96}, \cite[§2.4.2]{Gauts04}, \cite[§2.7]{Ism05}, \cite{U-UCMP69},
and \cite{Z-JCAM97}). It is also well known that $Q_{n}^{c}(x)$ can be
represented as%
\begin{equation}
Q_{n}^{c}(x)=P_{n}(x)-r_{n-1}\,P_{n-1}(x),\quad n=0,1,2,\ldots ,
\label{[S1]-ConnForm-0}
\end{equation}%
where $Q_{0}^{c}(x)=1$,%
\begin{equation}
r_{n-1}=r_{n-1}(c)=\frac{F_{n}(c)}{F_{n-1}(c)},\quad c\notin E,
\label{[S1]-rn(c)}
\end{equation}%
and $F_{-1}(c)=1$. The functions%
\begin{equation*}
F_{n}(s)=\int_{E}\frac{P_{n}(x)}{x-s}d\mu (x),\quad s\in \mathbb{C}%
\,\diagdown \,E,
\end{equation*}%
are the Cauchy integrals of $\{P_{n}\}_{n\geq 0}$, or functions of the
second kind associated with the monic polynomials $\{P_{n}\}_{n\geq 0}$. For
a proper way to compute the above Cauchy integrals, we refer the reader to 
\cite[§ 2.3]{Gauts04}.

It is clear that%
\begin{equation}
K_{n}^{c}(x,y)=\sum\limits_{k=0}^{n}\frac{Q_{k}^{c}(x)Q_{k}^{c}(y)}{%
||Q_{k}^{c}||_{\nu }^{2}}=\frac{%
Q_{n+1}^{c}(x)Q_{n}^{c}(y)-Q_{n+1}^{c}(y)Q_{n}^{c}(x)}{(x-y)}\frac{1}{%
||Q_{n}^{c}||_{\nu }^{2}}  \label{[S1]-Kcn(xy)}
\end{equation}%
are the kernel polynomials corresponding to the MOPS $\{Q_{n}^{c}\}_{n\geq
0} $, which also satisfies the corresponding reproducing property of
polynomial kernels with respect to the measure $d\nu $%
\begin{equation}
\int_{E}f\left( x\right) K_{n}^{c}\left( x,c\right) d\nu (x)=f\left(
c\right) ,  \label{[S1]-RepPropKc}
\end{equation}%
for any polynomial $f\in \mathbb{P}$\ with $\deg f\leq n$. The so called
confluent form of (\ref{[S1]-Kcn(xy)}) is given by (see \cite{Chi78})%
\begin{equation}
K_{n}^{c}(c,c)=\dfrac{[Q_{n+1}^{c}]^{\prime
}(c)Q_{n}^{c}(c)-[Q_{n}^{c}]^{\prime }(c)Q_{n+1}^{c}(c)}{||Q_{n}^{c}||_{\nu
}^{2}},  \label{[S1]-KcnConfl}
\end{equation}%
which is always a positive quantity%
\begin{equation}
K_{n}^{c}(c,c)=\sum_{k=0}^{n}\frac{[Q_{k}^{c}(c)]^{2}}{\left\Vert
Q_{k}^{c}\right\Vert _{\nu }^{2}}>0.  \label{[S1]-PositiveKcn}
\end{equation}%
The key concept to find several of our results is that the polynomials $%
\{P_{n}\}_{n\geq 0}$ are the monic kernel polynomials of parameter $c$ of
the sequence $\{Q_{n}^{c}\}_{n\geq 0}$. According to this argument, the
following expressions%
\begin{equation}
P_{n}(x)=\frac{\Vert Q_{n}^{c}\Vert _{\nu }^{2}}{Q_{n}^{c}(c)}K_{n}^{c}(x,c)=%
\frac{1}{(x-c)}\left( Q_{n+1}^{c}(x)-\frac{Q_{n+1}^{c}(c)}{Q_{n}^{c}(c)}%
Q_{n}^{c}(x)\right)  \label{[S1]-PnmonicKc}
\end{equation}%
hold.

Finally, let $\{{Q_{n}^{c,N}\}}_{n\geq 0}$ be the MOPS associated to $d\nu
_{N}$ when $N>0$. That is, $\{{Q_{n}^{c,N}\}}_{n\geq 0}$ are the Geronimus
perturbed polynomials orthogonal with respect to the the inner product%
\begin{equation}
\langle f,g\rangle _{\nu _{N}}=\int_{E}f(x)g(x)\frac{1}{(x-c)}d\mu
(x)+Nf(c)g(c).  \label{[S1]-InnProd-4}
\end{equation}%
Note that this is a standard inner product in the sense that, for every $%
f,g\in \mathbb{P}$, we have $\langle xf,g\rangle _{\nu _{N}}=\langle
f,xg\rangle _{\nu _{N}}$. From (\ref{[S1]-InnProd-3}) and (\ref%
{[S1]-InnProd-4}), a trivial verification shows that%
\begin{equation}
\langle f,g\rangle _{\nu _{N}}=\langle f,g\rangle _{\nu }+Nf(c)g(c).
\label{[S1]-InnProd-Rel-2}
\end{equation}

Is the aim of this contribution to find and analyze the asymptotic behavior
of the zeros of ${Q_{n}^{c,N}(x)}$ with the parameter $N$, present in the
Geronimus perturbation (\ref{[S1]-GeronimusTr}), and provide as well an
electrostatic model for these zeros when the original measure $\mu $\ is
semiclassical. To this end, we will use some remarkable facts, which are
straightforward consequences of the inner products (\ref{[S1]-InnProd-1}), (%
\ref{[S1]-InnProd-2}), (\ref{[S1]-InnProd-3}) and (\ref{[S1]-InnProd-4}).
Taking into account that the multiplication operator by $(x-c)$ is a
symmetric operator with respect to (\ref{[S1]-InnProd-3}), for any $%
f(x),g(x)\in \mathbb{P}$ we have%
\begin{equation*}
\langle (x-c)f,g\rangle _{\nu }=\langle f,(x-c)g\rangle _{\nu }=\langle
f,g\rangle _{\mu }
\end{equation*}%
If we consider the polynomials $(x-c)f(x)$ or $(x-c)g(x)$ in the above
expression, we deduce%
\begin{equation*}
(x-c)f(x)|_{x=c}=(x-c)g(x)|_{x=c}=0,
\end{equation*}%
which makes it obvious that $(x-c)$ is also a symmetric operator with
respect to the Geronimus inner product (\ref{[S1]-InnProd-4}), i.e.,%
\begin{equation}
\langle (x-c)f,g\rangle _{\nu _{N}}=\langle f,(x-c)g\rangle _{\nu
_{N}}=\langle (x-c)f,g\rangle _{\nu }.  \label{[S1]-InnProd-Rel-3}
\end{equation}%
Finally, another useful consequence of the above relations is%
\begin{equation}
\langle (x-c)f,(x-c)g\rangle _{\nu _{N}}=\langle f,g\rangle _{c,[1]}.
\label{[S1]-InnProd-Rel-4}
\end{equation}


\section{Statement of the main results}

\label{[SECTION-2]-MainRes}



\subsection{Connection formulas}


Next, we provide a new connection formula for the Geronimus perturbed
orthogonal polynomials $Q_{n}^{c,N}(x)$, in terms of the polynomials $%
Q_{n}^{c}(x)$ and the monic Kernel polynomials $P_{n}^{c,[1]}(x)$. This
representation will allow us to obtain the results about monotonicity,
asymptotics, and speed of convergence (presented below in this Section) for
the zeros of $Q_{n}^{c,N}(x)$ in terms of the parameter $N$ present in the
perturbation (\ref{[S1]-GeronimusTr}).


\begin{theorem}[connection formula]
\label{[S1]-THEO-1}The Geronimus perturbed MOPS $\{{\tilde{Q}_{n}^{c,N}\}}%
_{n\geq 0}$\ can be represented as%
\begin{equation}
\tilde{Q}_{n}^{c,N}(x)=Q_{n}^{c}(x)+NB_{n}^{c}(x-c)P_{n-1}^{c,[1]}(x),
\label{[S2]-ConnForm-Main}
\end{equation}%
with ${\tilde{Q}_{n}^{c,N}(x)=\kappa }_{n}{Q_{n}^{c,N}(x)}$, ${\kappa }%
_{n}=1+NB_{n}^{c}$ and%
\begin{equation}
B_{n}^{c}=\frac{-Q_{n}^{c}(c)P_{n-1}(c)}{\Vert P_{n-1}\Vert _{\mu }^{2}}%
=K_{n-1}^{c}(c,c)>0.  \label{[S2]-Bcn}
\end{equation}
\end{theorem}


Observe that one can even give another alternative expression for $B_{n}^{c}$%
, which only involves polynomials and functions of the second kind relative
to the original measure $\mu $, evaluated at the point of tranformation $c$.
Combining (\ref{[S1]-ConnForm-0}) with (\ref{[S2]-Bcn}), we deduce that%
\begin{equation}
B_{n}^{c}=K_{n-1}^{c}(c,c)=\frac{r_{n-1}\,P_{n-1}^{2}-P_{n}(c)P_{n-1}(c)}{%
\Vert P_{n-1}\Vert _{\mu }^{2}}.  \label{[S2]-Bcn-3}
\end{equation}

As a direct consequence of the above theorem, we can express ${Q_{n}^{c,N}(x)%
}$ in terms of only two consecutive elements of the initial sequence $%
\{P_{n}\}_{n\geq 0}$. This expression of ${Q_{n}^{c,N}(x)}$\ was already
studied in the literature (see \cite[formula (1.4)]{M-PMH90} and \cite[Sec. 1%
]{DM-NA13}). In fact, the original aim of Geronimus in its pioneer works on
the subject was to find necessary and sufficient conditions for the
existence of a sequence of coefficients $\Lambda _{n}$, such that the linear
combination of monic polinomials%
\begin{equation*}
P_{n}(x)+\Lambda _{n}P_{n-1}(x),\quad \Lambda _{n}\neq 0,\,n=1,2,\ldots ,
\end{equation*}%
were, in turn, orthogonal with respect to some measure supported on $\mathbb{%
R}$. Here we rewrite the value of $\Lambda _{n}$ in several new equivalent
ways. Substituting (\ref{[S1]-ConnForm-0}) and (\ref{[S1]-Pc1n}) into (\ref%
{[S2]-ConnForm-Main}) yields%
\begin{equation*}
\tilde{Q}_{n}^{c,N}(x)={\kappa }%
_{n}Q_{n}^{c,N}(x)=P_{n}(x)-r_{n-1}P_{n-1}(x)+NB_{n}^{c}\left( P_{n}(x)-\pi
_{n-1}P_{n-1}(x)\right) .
\end{equation*}%
Thus, having in mind that ${\kappa }_{n}=1+NB_{n}^{c}$, after some trivial
computations we can state the following result.


\begin{proposition}
\label{[S1]-PROP-1}The monic Geronimus perturbed orthogonal polynomials of
the sequence $\{{Q_{n}^{c,N}}\}_{n\geq 0}$ can be represented as%
\begin{equation}
{Q_{n}^{c,N}(x)}=P_{n}(x)+\Lambda _{n}^{c}\,P_{n-1}(x),
\label{[S2]-ConnForm-1}
\end{equation}%
with%
\begin{equation}
\Lambda _{n}^{c}=\Lambda _{n}^{c}(N)=\frac{\pi _{n-1}-r_{n-1}}{1+NB_{n}^{c}}%
-\pi _{n-1},  \label{[S2]-LambdacN}
\end{equation}%
and $\pi _{n-1}$, $r_{n-1}$ given respectively in (\ref{[S1]-param-pi}) and (%
\ref{[S1]-rn(c)}) respectively. Notice that $\Lambda _{n}^{c}$ is
independent of the variable $x$.
\end{proposition}


\begin{remark}
\label{[S2]-REMK-1}The coefficient $\Lambda _{n}^{c}(N)$ can also be
expressed\ only in terms of quantities relative to the original
non-perturbed measure $\mu $, the point of transformation $c$ and the mass $%
N $. Thus, from (\ref{[S2]-Bcn-3}) and (\ref{[S2]-LambdacN}), we obtain%
\begin{equation*}
\Lambda _{n}^{c}(N)=\left( \frac{1}{\pi _{n-1}-r_{n-1}}-N\frac{P_{n-1}^{2}(c)%
}{\Vert P_{n-1}\Vert _{\mu }^{2}}\right) ^{-1}-\pi _{n-1}.
\end{equation*}%
Also, observe that for $N=0$, the coefficient $\Lambda _{n}^{c}(0)$ reduces
to $-r_{n-1}$, and we recover the connection formula (\ref{[S1]-ConnForm-0}).
\end{remark}


\subsection{Asymptotic behavior and sharp limits of the zeros}


Let $x_{n,s}$, $x_{n,s}^{c,[k]}$, $y_{n,s}^{c}$, and $y_{n,s}^{c,N}$, $%
s=1,\ldots ,n$ be the zeros of $P_{n}(x)$, $P_{n}^{c,[k]}(x)$, ${Q_{n}^{c}(x)%
}$, and ${Q_{n}^{c,N}(x)}$, respectively, all arranged in an increasing
order, and assume that $C_{0}(E)=[a,b]$. Next, we analyze the behavior of
zeros $y_{n,s}^{c,N}$ as a function of the mass $N$ in (\ref%
{[S1]-GeronimusTr}). We obtain such a behavior when $N$ tends from zero to
infinity as well as we characterize the exact values of $N$ such the
smallest (respectively, the largest) zero of $\{{Q_{n}^{c,N}}\}_{n\geq 0}$
is located outside of $E=\mathrm{supp}(\mu )$.

In order to do that, we use a technique developed in \cite[Lemma 1]%
{BDR-JCAM02} and \cite[Lemmas 1 and 2]{DMR-ANM10} concerning the behavior
and the asymptotics of the zeros of linear combinations of two polynomials $%
h,g\in \mathbb{P}$ with interlacing zeros, such that $%
f(x)=h_{n}(x)+cg_{n}(x) $. From now on, we will refer to this technique as
the \textit{Interlacing Lemma}, and for the convenience of the reader we
include its statement in the final Appendix.

Taking into account that the positive constant $B_{n}^{c}$ does not depend
on $N$, we can use the connection formula (\ref{[S2]-ConnForm-Main}) to
obtain results about monotonicity, asymptotics, and speed of convergence for
the zeros of $Q_{n}^{c,N}(x)$ in terms of the mass $N$. Indeed, let assume
that $y_{n,k}^{c,N}$, $k=1,2,...,n,$ are the zeros of $Q_{n}^{c,N}(x)$.
Thus, from (\ref{[S2]-ConnForm-Main}), the positivity of $B_{n}^{c}$, and
Theorem \ref{[S3]-LEMA-2}, we are in the hypothesis of the Interlacing
Lemma, and we immediately conclude the following results.


\begin{theorem}
\label{[S2]-THEO-2} If $C_{0}(E)=[a,b]$ and $c<a$, then%
\begin{equation*}
c<y_{n,1}^{c,N}<y_{n,1}^{c}<x_{n-1,1}^{c,[1]}<y_{n,2}^{c,N}<y_{n,2}^{c}<%
\cdots <x_{n-1,n-1}^{c,[1]}<y_{n,n}^{c,N}<y_{n,n}^{c}.
\end{equation*}%
Moreover, each $y_{n,k}^{c,N}$ is a decreasing function of $N$ and, for each 
$k=1,\ldots ,n-1$,%
\begin{equation*}
\lim_{N\rightarrow \infty }y_{n,1}^{c,N}=c,\quad \lim_{N\rightarrow \infty
}y_{n,k+1}^{c,N}=x_{n-1,k}^{c,[1]}\,,
\end{equation*}%
as well as%
\begin{equation*}
\begin{array}{l}
\lim\limits_{N\rightarrow \infty }N[y_{n,1}^{c,N}-c]=\dfrac{-Q_{n}^{c}(c)}{%
B_{n}^{c}P_{n-1}^{c,[1]}(c)}, \\ 
\lim\limits_{N\rightarrow \infty }N[y_{n,k+1}^{c,N}-x_{n-1,k}^{c,[1]}]=%
\dfrac{-Q_{n}^{c}(x_{n-1,k}^{c,[1]})}{%
B_{n}^{c}(x_{n-1,k}^{c,[1]}-c)[P_{n-1}^{c,[1]}]^{\prime }(x_{n-1,k}^{c,[1]})}%
.%
\end{array}%
\end{equation*}
\end{theorem}


\begin{theorem}
\label{[S2]-THEO-3} If $C_{0}(E)=[a,b]$ and $c>b$, then%
\begin{equation*}
y_{n,1}^{c}<y_{n,1}^{c,N}<x_{n-1,1}^{c,[1]}<\cdots
<y_{n,n-1}^{c}<y_{n,n-1}^{c,N}<x_{n-1,n-1}^{c,[1]}<y_{n,n}^{c}<y_{n,n}^{c,N}<c.
\end{equation*}%
Moreover, each $y_{n,k}^{c,N}$ is an increasing function of $N$ and, for
each $k=1,\ldots ,n-1$,%
\begin{equation*}
\lim_{N\rightarrow \infty }y_{n,n}^{c,N}=c,\quad \lim_{N\rightarrow \infty
}y_{n,k}^{c,N}=x_{n-1,k}^{c,[1]},
\end{equation*}%
and 
\begin{equation*}
\begin{array}{l}
\lim\limits_{N\rightarrow \infty }N[c-y_{n,n}^{c,N}]=\dfrac{Q_{n}^{c}(c)}{%
B_{n}^{c}P_{n-1}^{c,[1]}(c)}, \\ 
\lim\limits_{N\rightarrow \infty }N[x_{n-1,k}^{c,[1]}-y_{n,k}^{c,N}]=\dfrac{%
Q_{n}^{c}(x_{n-1,k}^{c,[1]})}{%
B_{n}^{c}(x_{n-1,k}^{c,[1]}-c)[P_{n-1}^{c,[1]}]^{\prime }(x_{n-1,k}^{c,[1]})}%
.%
\end{array}%
\end{equation*}
\end{theorem}


Notice that the mass point $c$ attracts one zero of $Q_{n}^{c,N}(x)$, i.e.
when $N\rightarrow \infty $, it captures either the smallest or the largest
zero, according to the location of the point $c$ with respect to the support
of the measure $\mu $. When either $c<a$ or $c>b$, at most one of the zeros
of $Q_{n}^{c,N}(x)$ is located outside of $[a,b]$. Next, give explicitly the
value $N_{0}$ of the mass $N$, such that for $N>N_{0}$ one of the zeros is
located outside $[a,b]$.


\begin{corollary}[minimum mass]
If $C_{0}(E)=[a,b]$ and $c\notin \lbrack a,b]$, the following expressions
hold.

\begin{itemize}
\item[$(a)$] If $c<a$, then the smallest zero $y_{n,1}^{c,N}$ satisfies%
\begin{equation*}
\begin{array}{c}
y_{n,1}^{c,N}>a,\quad \mathrm{for}\quad N<N_{0}, \\ 
y_{n,1}^{c,N}=a,\quad \mathrm{for}\quad N=N_{0}, \\ 
y_{n,1}^{c,N}<a,\quad \mathrm{for}\quad N>N_{0},%
\end{array}%
\end{equation*}%
where%
\begin{equation*}
N_{0}=N_{0}(n,c,a)=\frac{-Q_{n}^{c}(a)}{K_{n-1}^{c}\left( c,c\right)
(a-c)P_{n-1}^{c,[1]}(a)}>0.
\end{equation*}

\item[$(b)$] If $c>b$, then the largest zero $y_{n,n}^{c,N}$ satisfies%
\begin{equation*}
\begin{array}{c}
y_{n,n}^{c,N}<b,\quad \mathrm{for}\quad N<N_{0}, \\ 
y_{n,n}^{c,N}=b,\quad \mathrm{for}\quad N=N_{0}, \\ 
y_{n,n}^{c,N}>b,\quad \mathrm{for}\quad N>N_{0},%
\end{array}%
\end{equation*}%
where%
\begin{equation*}
N_{0}=N_{0}(n,c,b)=\frac{-Q_{n}^{c}(b)}{K_{n-1}^{c}\left( c,c\right)
(b-c)P_{n-1}^{c,[1]}(b)}>0.
\end{equation*}
\end{itemize}
\end{corollary}


\begin{proof}
$(a)$\ \ In order to deduce the location of $y_{n,1}^{c,N}$ with respect to
the point $x=a$, it is enough to observe that $Q_{n}^{c,N}(a)=0$ if and only
if $N=N_{0}$.

$(b)$\ \ Also, in order to find the location of $y_{n,n}^{c,N}$ with respect
to the point $x=b$, notice that $Q_{n}^{c,N}(b)=0$ if and only if $N=N_{0}$.
\end{proof}


\subsection{Ladder operators and second order linear differential equation}


Our next result concerns the \textit{ladder (creation and annihilation)
operators}, and the \textit{second order linear differential equation} (also
known as the \textit{holonomic equation}) corresponding to $%
\{Q_{n}^{c,N}\}_{n\geq 0}$. We restrict ourselves to the case in which $\mu $
is a classical or semi-classical measure, and therefore satisfying a
structure relation (see \cite{DM-ANM90} and \cite{Mar91}) as%
\begin{equation}
\sigma (x)[P_{n}(x)]^{\prime }=a(x;n)P_{n}(x)+b(x;n)P_{n-1}(x)
\label{[S2]-StructRelation}
\end{equation}%
where $a(x;n)$\ and $b(x;n)$\ are polynomials in the variable $x$, whose
fixed degree do not depend on $n$.

In order to obtain these results, we will follow a different approach as in 
\cite[Ch. 3]{Ism05}. Our technique is based on the connection formula (\ref%
{[S2]-ConnForm-1}) given in Proposition \ref{[S1]-PROP-1}, the three term
recurrence relation (\ref{[S1]-3TRR-Pn}) satisfied by $\{P_{n}\}_{n\geq 0}$,
and the structure relation (\ref{[S2]-StructRelation}). The results are
presented here and will be proved in Section \ref{[SECTION-3]-Proofs}.


\begin{theorem}[ladder operators]
\label{[S2]-THEO-4}Let $\mathfrak{a}_{n}$\ and $\mathfrak{a}_{n}^{\dag }$\
be the differential operators%
\begin{eqnarray*}
\mathfrak{a}_{n} &=&-\xi _{1}^{c}(x;n)\mathrm{I}+\mathrm{D}_{x}, \\
\mathfrak{a}_{n}^{\dag } &=&-\eta _{2}^{c}(x;n)\mathrm{I}+\mathrm{D}_{x},
\end{eqnarray*}%
where $\mathrm{I}$, $\mathrm{D}_{x}$\ are the identity and $x$-derivative
operators respectively, satisfying%
\begin{eqnarray}
\mathfrak{a}_{n}[Q_{n}^{c,N}(x)] &=&\eta _{1}^{c}(x;n)\,Q_{n-1}^{c,N}(x),
\label{[S2]-LoweringEq} \\
\mathfrak{a}_{n}^{\dag }[Q_{n-1}^{c,N}(x)] &=&\xi
_{2}^{c}(x;n)\,Q_{n}^{c,N}(x),  \label{[S2]-RaisingEq}
\end{eqnarray}%
with, for $k=1,2$%
\begin{eqnarray*}
\xi _{k}^{c}(x;n) &=&\frac{C_{k}(x;n)B_{2}(x;n)\,\gamma
_{n-1}+D_{k}(x;n)\Lambda _{n-1}^{c}}{\Delta (x;n)\,\gamma _{n-1}}, \\
\eta _{k}^{c}(x;n) &=&\frac{D_{k}(x;n)-C_{k}(x;n)\Lambda _{n}^{c}}{\Delta
(x;n)}.
\end{eqnarray*}%
\ 

In turn, all the above expressions are given only in terms of the
coefficients in (\ref{[S1]-3TRR-Pn}), (\ref{[S2]-StructRelation}), and (\ref%
{[S2]-ConnForm-1}) as follows%
\begin{eqnarray*}
B_{2}(x;n) &=&\Lambda _{n-1}^{c}\left( \frac{1}{\Lambda _{n-1}^{c}}+\frac{%
(x-\beta _{n-1})}{\gamma _{n-1}}\right) , \\
C_{1}(x;n) &=&\frac{1}{\sigma (x)}\left( a(x;n)-\Lambda _{n}^{c}\frac{%
b(x;n-1)}{\gamma _{n-1}}\right) , \\
D_{1}(x;n) &=&\frac{1}{\sigma (x)}\left( b(x;n)+\Lambda
_{n}^{c}\,b(x;n-1)\left( \frac{a(x;n-1)}{b(x;n-1)}+\frac{(x-\beta _{n-1})}{%
\gamma _{n-1}}\right) \right) , \\
C_{2}(x;n) &=&\dfrac{-\Lambda _{n-1}^{c}}{\sigma (x)}\left( \frac{a(x;n)}{%
\gamma _{n-1}}+\frac{b(x;n-1)}{\gamma _{n-1}}\left( \frac{1}{\Lambda
_{n-1}^{c}}+\frac{(x-\beta _{n-1})}{\gamma _{n-1}}\right) \right) , \\
D_{2}(x;n) &=&\dfrac{\Lambda _{n-1}^{c}}{\sigma (x)}\left[ \dfrac{\sigma
(x)-b(x;n)}{\gamma _{n-1}}\right. +b(x;n-1)\times \\
&&\quad \left. \left( \frac{a(x;n-1)}{b(x;n-1)}+\dfrac{(x-\beta _{n-1})}{%
\gamma _{n-1}}\right) \left( \dfrac{1}{\Lambda _{n-1}^{c}}+\dfrac{(x-\beta
_{n-1})}{\gamma _{n-1}}\right) \right] , \\
\Delta (x;n) &=&B_{2}(x;n)+\frac{\Lambda _{n}^{c}\,\Lambda _{n-1}^{c}}{%
\gamma _{n-1}},\quad \deg \Delta (x;n)=1.
\end{eqnarray*}%
Thus, $\mathfrak{a}_{n}$\ and $\mathfrak{a}_{n}^{\dag }$\ are\ respectively
lowering and raising operators associated to the Geronimus perturbed MOPS $%
\{Q_{n}^{c,N}\}_{n\geq 0}$.
\end{theorem}


For a deeper discusion of raising and lowering operators we refer the reader
to \cite[Ch. 3]{Ism05}. We next provide the second order linear differential
equation satisfied by the MOPS $\{{Q}_{n}^{c,N}\}_{n\geq 0}$ when the
measure $\mu $ is semi-classical (for definition of a semi-classical measure
see \cite{Mar91}). This is the main tool for the further electrostatic
interpretation of zeros.


\begin{theorem}[holonomic equation]
\label{[S2]-THEO-5} The Geronimus perturbed MOPS $\{Q_{n}^{c,N}\}_{n\geq 0}$%
\ satisfies the holonomic equation (second order linear differential
equation)%
\begin{equation}
\lbrack Q_{n}^{c,N}(x)]^{\prime \prime }+\mathcal{R}(x;n)[Q_{n}^{c,N}(x)]^{%
\prime }+\mathcal{S}(x;n)Q_{n}^{c,N}(x)=0,  \label{[S2]-2ndODE}
\end{equation}%
where%
\begin{eqnarray*}
\mathcal{R}(x;n) &=&-\left( \xi _{1}^{c}(x;n)+\eta _{2}^{c}(x;n)+\frac{[\eta
_{1}^{c}(x;n)]^{\prime }}{\eta _{1}^{c}(x;n)}\right) , \\
\mathcal{S}(x;n) &=&\xi _{1}^{c}(x;n)\eta _{2}^{c}(x;n)-\eta
_{1}^{c}(x;n)\xi _{2}^{c}(x;n) \\
&&\quad +\frac{\xi _{1}^{c}(x;n)[\eta _{1}^{c}(x;n)]^{\prime }-[\xi
_{1}^{c}(x;n)]^{\prime }\eta _{1}^{c}(x;n)}{\eta _{1}^{c}(x;n)}.
\end{eqnarray*}
\end{theorem}


\section{Proofs of the main results.}

\label{[SECTION-3]-Proofs} 


\subsection{Proof of Theorem \protect\ref{[S1]-THEO-1} and the positivity of 
$B_{n}^{c}$}


First, we need to prove the following lemma concerning a first way to
represent the Geronimus perturbed polynomials $Q_{n}^{c,N}(x)$, using the
kernels (\ref{[S1]-Kcn(xy)}).

\begin{lemma}
\label{[S3]-LEMA-1}Let $\{Q_{n}^{c,N}\}_{n\geq 0}$ and $\{Q_{n}^{c}\}_{n\geq
0}$ be the MOPS corresponding to the measures $d\nu _{N}$ and $d\nu (x)$
respectively. Then, the following connection formula holds%
\begin{equation}
Q_{n}^{c,N}(x)=Q_{n}^{c}(x)-NQ_{n}^{c,N}(c)K_{n-1}^{c}(x,c),
\label{[S3]-FConex-0}
\end{equation}%
where%
\begin{equation}
Q_{n}^{c,N}(c)=\frac{Q_{n}^{c}(c)}{1+NK_{n-1}^{c}(c,c)}=\kappa
_{n}^{-1}Q_{n}^{c}(c),  \label{[S3]-RelacQs}
\end{equation}%
${\kappa }_{n}=1+NB_{n}^{c}$, and $K_{n-1}^{c}(c,c)$ is given in (\ref%
{[S1]-KcnConfl}).
\end{lemma}

\begin{proof}
From (\ref{[S1]-InnProd-Rel-2}) it is trivial to express $Q_{n}^{c,N}(x)$ in
terms of the polynomials $Q_{n}^{c}(x)$%
\begin{equation}
Q_{n}^{c,N}(x)=\sum_{k=0}^{n}b_{n,k}Q_{n}^{c}(x),  \label{[S3]-expansion}
\end{equation}%
having%
\begin{equation*}
b_{n,k}=\frac{\langle Q_{i}^{c}(x),Q_{n}^{c,N}(x)\rangle _{\nu }}{\left\Vert
Q_{k}^{c}\right\Vert _{\nu }^{2}},\quad 0\leq k\leq n-1.
\end{equation*}%
Thus, (\ref{[S3]-expansion}) becomes%
\begin{equation*}
Q_{n}^{c,N}(x)=Q_{n}^{c}(x)-NQ_{n}^{c,N}(c)\sum_{k=0}^{n-1}\frac{%
Q_{k}^{c}(x)Q_{k}^{c}\left( c\right) }{\left\Vert Q_{k}^{c}\right\Vert _{\nu
}^{2}}.
\end{equation*}%
Next, taking into account (\ref{[S1]-Kernel-n}) for the sequence $%
\{Q_{k}^{c}\}_{n\geq 0}$, we get%
\begin{equation*}
Q_{n}^{c,N}(x)=Q_{n}^{c}(x)-NQ_{n}^{c,N}(c)K_{n-1}^{c}(x,c).
\end{equation*}%
In order to find $Q_{n}^{c,N}(c)$, we evaluate (\ref{[S3]-FConex-0}) in $x=c$%
. Thus%
\begin{equation}
Q_{n}^{c,N}(c)=\frac{Q_{n}^{c}(c)}{1+NK_{n-1}^{c}(c,c)}.  \label{[S3]-QncN}
\end{equation}%
This completes the proof.
\end{proof}


Next, in order to prove the orthogonality of the polynomials defined by (\ref%
{[S2]-ConnForm-Main}), we deal with the basis $\mathcal{B}%
^{n}=\{1,(x-c),(x-c)^{2},\ldots ,(x-c)^{n}\}$ of the space of polynomials of
degree at most $n$. We prove that there exist a positive constant $B_{n}^{c}$
such that every element in this basis is orthogonal to every polynomial of
the sequence$\{\tilde{Q}_{n}^{c,N}\}_{n\geq 0}$ with respect to the inner
product (\ref{[S1]-InnProd-4}). Thus, from (\ref{[S1]-InnProd-Rel-2}), (\ref%
{[S2]-ConnForm-Main}) and $\tilde{Q}_{n}^{c,N}(c)=\kappa _{n}Q_{n}^{c,N}(c)$
we have%
\begin{equation*}
\langle 1,\tilde{Q}_{n}^{c,N}\rangle _{\nu _{N}}=\langle
1,Q_{n}^{c}(x)\rangle _{\nu }+NB_{n}^{c}\langle
1,(x-c)P_{n-1}^{c,[1]}(x)\rangle _{\nu }+N\kappa _{n}Q_{n}^{n,N}(c)=0.
\end{equation*}%
Notice that to get $\langle 1,\tilde{Q}_{n}^{c}(x)\rangle _{\nu }=0$ for
every $n>1$ we need%
\begin{equation}
B_{n}^{c}=\frac{-\kappa _{n}Q_{n}^{c,N}(c)}{\langle
1,(x-c)P_{n-1}^{c,[1]}(x)\rangle _{\nu }}.  \label{[S3]-Bn1st}
\end{equation}%
Next, we prove the orthogonality with respect to the other elements of $%
\mathcal{B}^{n}$. From (\ref{[S2]-ConnForm-Main}), (\ref{[S1]-InnProd-Rel-3}%
), (\ref{[S1]-InnProd-Rel-4}) and the orthogonality with respect to $d\nu
(x) $ and $d\mu ^{\lbrack 1]}(x)$ we get%
\begin{equation*}
\langle (x-c),Q_{n}^{c,N}(x)\rangle _{\nu _{N}}=\langle
(x-c),Q_{n}^{c}(x)\rangle _{\nu }+NB_{n}^{c}\langle
1,P_{n-1}^{c,[1]}(x)\rangle _{\mu ,[1]}=0,\quad n>1.
\end{equation*}%
We continue in this fashion, verifying that%
\begin{eqnarray*}
\langle (x-c)^{n-1},\tilde{Q}_{n}^{c,N}(x)\rangle _{\nu _{N}} &=&\langle
(x-c)^{n-1},Q_{n}^{c}(x)\rangle _{\nu _{N}}+NB_{n}^{c}\langle
(x-c)^{n-1},(x-c)P_{n-1}^{c,[1]}(x)\rangle _{\nu _{N}} \\
&=&\langle (x-c)^{n-1},Q_{n}^{c}(x)\rangle _{\nu }+NB_{n}^{c}\langle
(x-c)^{n-2},P_{n-1}^{c,[1]}(x)\rangle _{\mu ,[1]}=0,
\end{eqnarray*}%
and finally%
\begin{eqnarray*}
\langle (x-c)^{n},\tilde{Q}_{n}^{c,N}\rangle _{\nu _{N}} &=&\Vert \tilde{Q}%
_{n}^{c,N}\Vert _{\nu _{N}}^{2} \\
&=&\langle (x-c)^{n},Q_{n}^{c}(x)\rangle _{\nu }+NB_{n}^{c}\langle
(x-c)^{n-1},P_{n-1}^{c,[1]}(x)\rangle _{_{\mu ,[1]}} \\
&=&\Vert Q_{n}^{c}\Vert _{\nu }^{2}+NB_{n}^{c}\Vert P_{n-1}^{c,[1]}\Vert
_{\mu ,[1]}^{2}.
\end{eqnarray*}%
Sumarizing%
\begin{eqnarray*}
\langle 1,\tilde{Q}_{n}^{c,N}\rangle _{\nu _{N}} &=&\langle
1,Q_{n}^{c}\rangle _{\nu }+NB_{n}^{c}\langle 1,P_{n-1}^{c,[1]}\rangle _{\mu
}+NQ_{n}^{c}(c)=0, \\
\langle (x-c),\tilde{Q}_{n}^{c,N}\rangle _{\nu _{N}} &=&\langle
(x-c),Q_{n}^{c}\rangle _{\nu }+NB_{n}^{c}\langle 1,P_{n-1}^{c,[1]}\rangle
_{\mu ,[1]}=0, \\
&&\vdots \\
\langle (x-c)^{n-1},\tilde{Q}_{n}^{c,N}\rangle _{\nu _{N}} &=&\langle
(x-c)^{n-1},Q_{n}^{c}\rangle _{\nu }+NB_{n}^{c}\langle
(x-c)^{n-2},P_{n-1}^{c,[1]}\rangle _{\mu ,[1]} \\
&=&0, \\
\langle (x-c)^{n},\tilde{Q}_{n}^{c,N}\rangle _{\nu _{N}} &=&\Vert
Q_{n}^{c}\Vert _{\nu }^{2}+NB_{n}^{c}\Vert P_{n-1}^{c,[1]}\Vert _{\mu
,[1]}^{2}.
\end{eqnarray*}


Next, we briefly analyze the value of $B_{n}^{c}$ in (\ref{[S3]-Bn1st}) and
we prove (\ref{[S2]-Bcn}). From (\ref{[S1]-Kernel-n}) and (\ref{[S1]-Pc1n})
we have%
\begin{eqnarray*}
\langle 1,(x-c)P_{n-1}^{c,[1]}\rangle _{\nu } &=&\langle
1,P_{n-1}^{c,[1]}(x)\rangle _{\mu }=\frac{\Vert P_{n-1}\Vert _{\mu }^{2}}{%
P_{n-1}(c)}\langle 1,K_{n-1}(x,c)\rangle _{\mu } \\
&=&\frac{\Vert P_{n-1}\Vert _{\mu }^{2}}{P_{n-1}(c)}\sum_{k=0}^{n-1}\frac{%
P_{k}(c)}{||P_{k}||_{\mu }^{2}}\langle 1,P_{k}(x)\rangle _{\mu }.
\end{eqnarray*}%
Because the orthogonality, the only term which survive in the above sum is
for $k=0$, hence%
\begin{equation*}
\sum_{k=0}^{n-1}\frac{P_{k}(c)}{||P_{k}||_{\mu }^{2}}\langle
1,P_{k}(x)\rangle _{\mu }=1,
\end{equation*}%
and therefore%
\begin{equation*}
\langle 1,(x-c)P_{n-1}^{c,[1]}\rangle _{\nu }=\frac{\Vert P_{n-1}\Vert _{\mu
}^{2}}{P_{n-1}(c)}.
\end{equation*}%
Thus, taking into account (\ref{[S3]-RelacQs})%
\begin{equation}
B_{n}^{c}=\frac{-Q_{n}^{c}(c)P_{n-1}(c)}{\Vert P_{n-1}\Vert _{\mu }^{2}}.
\label{[S3]-BnValue}
\end{equation}%
In order to prove (\ref{[S2]-Bcn}), from (\ref{[S1]-Pc1n}), (\ref%
{[S1]-PnmonicKc}), (\ref{[S1]-RepPropKc}) we get%
\begin{eqnarray}
\langle (x-c),P_{n-1}^{c,[1]}(x)\rangle _{\nu } &=&\int_{E}\left( x-c\right) 
\frac{1}{(x-c)}\left( P_{n+1}(x)-\frac{P_{n+1}(c)}{P_{n}(c)}P_{n}(x)\right)
d\nu (x)  \notag \\
&=&\int_{E}P_{n}(x)d\nu (x)-\frac{P_{n}(c)}{P_{n-1}(c)}\int_{E}P_{n-1}(x)d%
\nu (x)  \notag \\
&=&\frac{\Vert Q_{n}^{c}\Vert _{\nu }^{2}}{Q_{n}^{c}(c)}%
\int_{E}K_{n}^{c}(x,c)d\nu (x)  \notag \\
&&-\frac{Q_{n-1}^{c}(c)}{\Vert Q_{n-1}^{c}\Vert _{\nu }^{2}}\frac{\Vert
Q_{n}^{c}\Vert _{\nu }^{2}}{Q_{n}^{c}(c)}\frac{K_{n}^{c}(c,c)}{%
K_{n-1}^{c}(c,c)}\frac{\Vert Q_{n-1}^{c}\Vert _{\nu }^{2}}{Q_{n-1}^{c}(c)}%
\int_{E}K_{n-1}^{c}(x,c)d\nu (x)  \notag \\
&=&\frac{\Vert Q_{n}^{c}\Vert _{\nu }^{2}}{Q_{n}^{c}(c)}-\frac{Q_{n-1}^{c}(c)%
}{\Vert Q_{n-1}^{c}\Vert _{\nu }^{2}}\frac{\Vert Q_{n}^{c}\Vert _{\nu }^{2}}{%
Q_{n}^{c}(c)}\frac{K_{n}^{c}(c,c)}{K_{n-1}^{c}(c,c)}\frac{\Vert
Q_{n-1}^{c}\Vert _{\nu }^{2}}{Q_{n-1}^{c}(c)}  \notag \\
&=&\frac{\Vert Q_{n}^{c}\Vert _{\nu }^{2}}{Q_{n}^{c}(c)}\left( 1-\frac{%
K_{n}^{c}(c,c)}{K_{n-1}^{c}(c,c)}\right) .  \label{[S3]-Buffer}
\end{eqnarray}%
A general property for kernels is, from (\ref{[S1]-Kcn(xy)})%
\begin{equation*}
K_{n}^{c}(x,c)=\sum\limits_{k=0}^{n}\frac{Q_{k}^{c}(x)Q_{k}^{c}(c)}{%
||Q_{k}^{c}||_{\nu }^{2}}=\frac{Q_{n}^{c}(x)Q_{n}^{c}(c)}{||Q_{n}^{c}||_{\nu
}^{2}}+K_{n-1}^{c}(x,c)
\end{equation*}%
and therefore%
\begin{equation}
\left( 1-\frac{K_{n}^{c}(c,c)}{K_{n-1}^{c}(c,c)}\right) =\frac{%
-[Q_{n}^{c}(c)]^{2}}{||Q_{n}^{c}||_{\nu }^{2}K_{n-1}^{c}(c,c)}
\label{[S3]-1mratioKc}
\end{equation}%
Replacing in (\ref{[S3]-Buffer})%
\begin{equation*}
\langle (x-c),P_{n-1}^{c,[1]}(x)\rangle _{\nu }=\frac{\Vert Q_{n}^{c}\Vert
_{\nu }^{2}}{Q_{n}^{c}(c)}\left( \frac{-[Q_{n}^{c}(c)]^{2}}{%
||Q_{n}^{c}||_{\nu }^{2}K_{n-1}^{c}(c,c)}\right) =\frac{-Q_{n}^{c}(c)}{%
K_{n-1}^{c}(c,c)}.
\end{equation*}%
Thus, from (\ref{[S3]-QncN}) and (\ref{[S3]-Bn1st})%
\begin{equation*}
B_{n}^{c}=\frac{-\kappa _{n}Q_{n}^{c,N}(c)}{\langle
1,(x-c)P_{n-1}^{c,[1]}(x)\rangle _{\nu }}=\frac{-\kappa _{n}\frac{%
Q_{n}^{c}(c)}{1+NK_{n-1}^{c}(c,c)}}{\frac{-Q_{n}^{c}(c)}{K_{n-1}^{c}(c,c)}}%
=K_{n-1}^{c}(c,c).
\end{equation*}%
Finally, being $c\notin \mathrm{supp}(\nu) $, from (\ref{[S1]-PositiveKcn})
we can conclude that $B_{n}^{c}$ is always positive 
\begin{equation*}
B_{n}^{c}=K_{n-1}^{c}(c,c)>0.
\end{equation*}


\subsection{Proofs of Theorems \textbf{\protect\ref{[S2]-THEO-2} and \protect
\ref{[S2]-THEO-3}}}


To apply the Interlacing Lemma and get the results of Theorems \ref%
{[S2]-THEO-2} and \ref{[S2]-THEO-3}, we need to show that we satisfy the
hypotheses of the Interlacing Lemma. To do this, we first prove that the
zeros of $Q_{n}^{c}(x)$\ and $(x-c)P_{n-1}^{c,[1]}(x)$ interlace.


\begin{lemma}
\label{[S3]-LEMA-2} Let $y_{n,k}^{c}$ and $x_{n,k}^{c,[1]}$ be the zeros of $%
Q_{n}^{c}(x)$ and $P_{n}^{c,[1]}(x)$, respectively, all arranged in an
increasing order. The inequalities%
\begin{equation*}
y_{n+1,1}^{c}<x_{n,1}^{c,[1]}<y_{n+1,2}^{c}<x_{n,2}^{c,[1]}<\cdots
<y_{n+1,n}^{c}<x_{n,n}^{c,[1]}<y_{n+1,n+1}^{c}
\end{equation*}%
hold for every $n\in \mathbb{N}$.
\end{lemma}


\begin{proof}
Combining (\ref{[S1]-PnmonicKc}) with (\ref{[S1]-Pc1n}) yields%
\begin{equation}
(x-c)^{2}P_{n}^{c,[1]}(x)=Q_{n+2}^{c}(x)-d_{n}^{c}Q_{n+1}^{c}(x)+e_{n}^{c}Q_{n}^{c}(x),
\label{[S3]-ConnF-PQcn-1}
\end{equation}%
where%
\begin{eqnarray*}
e_{n}^{c} &=&\frac{P_{n+1}(c)}{P_{n}(c)}\frac{Q_{n+1}^{c}(c)}{Q_{n}^{c}(c)}
\\
&=&\frac{\Vert Q_{n+1}^{c}\Vert _{\nu }^{2}}{\Vert Q_{n}^{c}\Vert _{\nu }^{2}%
}\frac{K_{n+1}^{c}(c,c)}{K_{n}^{c}(c,c)}>0,
\end{eqnarray*}%
and%
\begin{eqnarray*}
d_{n}^{c} &=&\frac{Q_{n+2}^{c}(c)}{Q_{n+1}^{c}(c)}+\frac{P_{n+1}(c)}{P_{n}(c)%
} \\
&=&\frac{Q_{n+2}^{c}(c)}{Q_{n+1}^{c}(c)}+\frac{Q_{n}^{c}(c)}{Q_{n+1}^{c}(c)}%
\frac{\Vert Q_{n+1}^{c}\Vert _{\nu }^{2}}{\Vert Q_{n}^{c}\Vert _{\nu }^{2}}%
\frac{K_{n+1}^{c}(c,c)}{K_{n}^{c}(c,c)} \\
&=&\frac{Q_{n+2}^{c}(c)+Q_{n}^{c}(c)e_{n}^{c}}{Q_{n+1}^{c}(c)}.
\end{eqnarray*}%
On the other hand, the sequence $\{Q_{n}^{c}\}_{n\geq 0}$ satisfies the
three term recurrence relation%
\begin{equation}
Q_{n}^{c}(x)=(x-\beta _{n}^{c})Q_{n-1}^{c}(x)-\gamma
_{n}^{c}Q_{n-2}^{c}(x),\quad n=1,2,\ldots  \label{[S3]-TTRRQcn}
\end{equation}%
The coefficients $\beta _{n}^{c}$, and $\gamma _{n}^{c}$ are given in
several works. A particularly clear discussion about how to obtain $\beta
_{n}^{c}$, $\gamma _{n}^{c}$ from those $\beta _{n}$, $\gamma _{n}$ of the
initial $\mu $ is given in \cite[$§2.4.4$]{Gauts04}. From (\ref{[S1]-rn(c)}%
), for $n\geq 1$, the modified coefficients are given by%
\begin{eqnarray*}
\beta _{n}^{c} &=&\beta _{n}+r_{n}-r_{n-1}, \\
\gamma _{n}^{c} &=&\gamma _{n-1}\frac{r_{n-1}}{r_{n-2}},
\end{eqnarray*}%
with the initial convention, for $n=0$,%
\begin{eqnarray*}
\beta _{0}^{c} &=&\beta _{0}+r_{0}, \\
\gamma _{0}^{c} &=&\int_{E}d\nu (x)=\int_{E}\frac{1}{x-c}d\mu (x)=-F_{0}(c).
\end{eqnarray*}%
Combining (\ref{[S3]-ConnF-PQcn-1}) with (\ref{[S3]-TTRRQcn}) yields%
\begin{equation}
(x-c)^{2}P_{n}^{c,[1]}(x)=\left( x-\beta _{n+2}^{c}-d_{n}^{c}\right)
Q_{n+1}^{c}(x)+\left( e_{n}^{c}-\gamma _{n+2}^{c}\right) Q_{n}^{c}(x).
\label{[S3]-CF-PQcn-2}
\end{equation}%
Being $\mu $ a positive definite measure, the modified measure $d\nu $ is
also positive definite, because $c\notin E= \mathrm{supp}(\mu)$ and
therefore $(x-c)^{-1}$ do not change sign in $E$. Hence, by \cite[Th. 4.2 (a)%
]{Chi78} and (\ref{[S3]-1mratioKc}), the coefficient of $Q_{n}^{c}(x)$ in (%
\ref{[S3]-CF-PQcn-2}) can be expressed by%
\begin{eqnarray}
e_{n}^{c}-\gamma _{n+2}^{c} &=&\frac{\Vert Q_{n+1}^{c}\Vert _{\nu }^{2}}{%
\Vert Q_{n}^{c}\Vert _{\nu }^{2}}\frac{K_{n+1}^{c}(c,c)}{K_{n}^{c}(c,c)}-%
\frac{||Q_{n+1}^{c}||_{\nu }^{2}}{||Q_{n}^{c}||_{\nu }^{2}}  \notag \\
&=&\frac{||Q_{n+1}^{c}||_{\nu }^{2}}{||Q_{n}^{c}||_{\nu }^{2}}\left( \frac{%
K_{n+1}^{c}(c,c)}{K_{n}^{c}(c,c)}-1\right)  \label{[S3]-coefQcnPos} \\
&=&\frac{1}{||Q_{n}^{c}||_{\nu }^{2}}\frac{[Q_{n+1}^{c}(c)]^{2}}{%
K_{n}^{c}(c,c)}>0,  \notag
\end{eqnarray}%
which is positive for every $n\geq 0$, no matter the position of $c$ with
respect to the interval $E$.

Finally, evaluating $P_{n}^{c,[1]}(x)$ at the zeros $y_{n+1,k}^{c}$, from (%
\ref{[S3]-CF-PQcn-2}) and (\ref{[S3]-coefQcnPos}), we get%
\begin{equation*}
(x-c)^{2}P_{n}^{c,[1]}(y_{n+1,k}^{c})=\left( e_{n}^{c}-\gamma
_{n+2}^{c}\right) Q_{n}^{c}(y_{n+1,k}^{c}),\quad k=1,\ldots ,n+1,
\end{equation*}%
so it is clear that%
\begin{equation}
sign(P_{n}^{c,[1]}(y_{n+1,k}^{c}))=sign(Q_{n}^{c}(y_{n+1,k}^{c})),\quad
k=1,\ldots ,n+1.  \label{[S3]-Prop2}
\end{equation}%
Thus, from (\ref{[S3]-Prop2}) and the very well known fact that the zeros of 
$Q_{n+1}^{c}(x)$\ interlace with the zeros of $Q_{n}^{c}(x)$, we conclude
that $P_{n}^{c,[1]}(x)$ has at least one zero in every interval $%
(y_{n+1,k}^{c},y_{n+1,k+1}^{c})$\ for every $k=1,\ldots n$. This completes
the proof.
\end{proof}



\subsection{Proofs of Theorems \protect\ref{[S2]-THEO-4} and \protect\ref%
{[S2]-THEO-5}}


We begin by proving several lemmas that are needed for the proof of Theorem %
\ref{[S2]-THEO-4}.


\begin{lemma}
\label{[S3]-LEMA-3}For the MOPS $\{Q_{n}^{c,N}\}_{n\geq 0}$ and $%
\{P_{n}\}_{n\geq 0}$ we have%
\begin{equation}
\lbrack {Q_{n}^{c,N}(x)}]^{\prime }=C_{1}(x;n)P_{n}(x)+D_{1}(x;n)P_{n-1}(x),
\label{[S3]-DerQ-C1D1}
\end{equation}%
where%
\begin{eqnarray}
C_{1}(x;n) &=&\frac{1}{\sigma (x)}\left( a(x;n)-\Lambda _{n}^{c}\frac{%
b(x;n-1)}{\gamma _{n-1}}\right) ,  \label{[S3]-Coefs-ABCD1} \\
D_{1}(x;n) &=&\frac{1}{\sigma (x)}\left( b(x;n)+\Lambda
_{n}^{c}\,b(x;n-1)\left( \frac{a(x;n-1)}{b(x;n-1)}+\frac{(x-\beta _{n-1})}{%
\gamma _{n-1}}\right) \right) .  \notag
\end{eqnarray}%
The coefficient $\Lambda _{n}^{c}$\ is given in (\ref{[S2]-LambdacN}), $%
\beta _{n-1}$, $\gamma _{n-1}$ are given in (\ref{[S1]-3TRR-Pn})\ and $%
\sigma (x)$, $a(x;n)$, $b(x;n)$ come from the structure relation (\ref%
{[S2]-StructRelation}) satisfied by $\{P_{n}\}_{n\geq 0}$.
\end{lemma}



\begin{proof}
Shifting the index in (\ref{[S2]-StructRelation}) as $n\rightarrow n-1$, and
using (\ref{[S1]-3TRR-Pn}) we obtain%
\begin{equation}
\lbrack P_{n-1}(x)]^{\prime }=\frac{-b(x;n-1)}{\sigma (x)\,\gamma _{n-1}}%
P_{n}(x)+\left( \frac{a(x;n-1)}{\sigma (x)}+\frac{b(x;n-1)(x-\beta _{n-1})}{%
\sigma (x)\,\gamma _{n-1}}\right) P_{n-1}(x).  \label{[S3]-Pnm1Der}
\end{equation}%
Next, taking $x$ derivative in both sides of (\ref{[S2]-ConnForm-1}), we get%
\begin{equation*}
\lbrack {Q_{n}^{c,N}(x)]}^{\prime }=[P_{n}(x)]^{\prime }+\Lambda
_{n}^{c}\,[P_{n-1}(x)]^{\prime }.
\end{equation*}%
Substituting (\ref{[S2]-StructRelation}) and (\ref{[S3]-Pnm1Der}) into the
above expression the Lemma follows.
\end{proof}



\begin{lemma}
\label{[S3]-LEMA-4}The sequences of monic polynomials $\{Q_{n}^{c,N}\}_{n%
\geq 0}$ and $\{P_{n}\}_{n\geq 0}$ are also related by%
\begin{eqnarray}
{Q_{n-1}^{c,N}(x)} &=&A_{2}(n)P_{n}(x)+B_{2}(x;n)P_{n-1}(x),
\label{[S3]-Qnm1-A2D2} \\
\lbrack Q_{n-1}^{c,N}(x)]^{\prime }
&=&C_{2}(x;n)P_{n}(x)+D_{2}(x;n)P_{n-1}(x),  \label{[S3]-DxQnm1-C2D2}
\end{eqnarray}%
where%
\begin{eqnarray}
A_{2}(n) &=&\dfrac{-\Lambda _{n}^{c}}{\gamma _{n-1}},  \notag \\
B_{2}(x;n) &=&\Lambda _{n-1}^{c}\left( \frac{1}{\Lambda _{n-1}^{c}}+\frac{%
(x-\beta _{n-1})}{\gamma _{n-1}}\right) ,  \notag \\
C_{2}(x;n) &=&-\dfrac{\Lambda _{n-1}^{c}}{\sigma (x)}\left( \frac{a(x;n)}{%
\gamma _{n-1}}+\frac{b(x;n-1)}{\gamma _{n-1}}\left( \frac{1}{\Lambda
_{n-1}^{c}}+\frac{(x-\beta _{n-1})}{\gamma _{n-1}}\right) \right)
\label{[S3]-Coefs-ABCD2} \\
D_{2}(x;n) &=&\dfrac{\Lambda _{n-1}^{c}}{\sigma (x)}\left[ \dfrac{\sigma
(x)-b(x;n)}{\gamma _{n-1}}\right. +b(x;n-1)\times  \notag \\
&&\quad \left. \left( \frac{a(x;n-1)}{b(x;n-1)}+\dfrac{(x-\beta _{n-1})}{%
\gamma _{n-1}}\right) \left( \dfrac{1}{\Lambda _{n-1}^{c}}+\dfrac{(x-\beta
_{n-1})}{\gamma _{n-1}}\right) \right]  \notag
\end{eqnarray}
\end{lemma}


\begin{proof}
The proof of (\ref{[S3]-Qnm1-A2D2}) and (\ref{[S3]-DxQnm1-C2D2}) is a
straightforward consequence of (\ref{[S2]-ConnForm-1}), (\ref%
{[S2]-StructRelation}), Lemma \ref{[S3]-LEMA-3}, and the three term
recurrence relation (\ref{[S1]-3TRR-Pn}) for the MOPS $\{P_{n}\}_{n\geq 0}$.
\end{proof}


\begin{remark}
\label{[S3]-REMK-1} Observe that the set of coefficients (\ref%
{[S3]-Coefs-ABCD1}) and (\ref{[S3]-Coefs-ABCD2}) can be given strictly in
terms of the following known quantities: the coefficient $\Lambda _{n}^{c}$\
in (\ref{[S2]-LambdacN}), the coefficients $\beta _{n-1}$, $\gamma _{n-1}$
of the three term recurrence relation (\ref{[S1]-3TRR-Pn})\ and $\sigma (x)$%
, $a(x;n)$, $b(x;n)$ of the structure relation (\ref{[S2]-StructRelation})
satisfied by $\{P_{n}\}_{n\geq 0}$.
\end{remark}


The following lemma shows the converse relation of (\ref{[S2]-ConnForm-1})--(%
\ref{[S3]-Qnm1-A2D2}) for the polynomials $P_{n}(x)$ and $P_{n-1}(x)$. That
is, we express these two consecutive polynomials of $\{P_{n}\}_{n\geq 0}$ in
terms of only two consecutive Geronimus perturbed polynomials of the MOPS $%
\{Q_{n}^{c,N}\}_{n\geq 0}$.


\begin{lemma}
\label{[S3]-LEMA-5}%
\begin{eqnarray}
P_{n}(x) &=&\frac{B_{2}(x;n)}{\Delta (x;n)}Q_{n}^{c,N}(x)-\frac{\Lambda
_{n}^{c}}{\Delta (x;n)}Q_{n-1}^{c,N}(x),  \label{[S3]-InvR-Pn} \\
P_{n-1}(x) &=&\frac{\Lambda _{n-1}^{c}}{\Delta (x;n)\,\gamma _{n-1}}%
Q_{n}^{c,N}(x)+\frac{1}{\Delta (x;n)}Q_{n-1}^{c,N}(x).
\label{[S3]-InvR-Pnm1}
\end{eqnarray}%
where%
\begin{equation*}
\Delta (x;n)=\frac{\Lambda _{n-1}^{c}}{\gamma _{n-1}}\left( x-\beta
_{n-1}+\Lambda _{n}^{c}+\frac{\gamma _{n-1}}{\Lambda _{n-1}^{c}}\right)
,\quad \deg \Delta (x;n)=1.
\end{equation*}
\end{lemma}


\begin{proof}
Note that (\ref{[S2]-ConnForm-1})--(\ref{[S3]-Qnm1-A2D2}) can be interpreted
as a system of two linear equations with two polynomial unknowns, namely $%
P_{n}(x)$ and $P_{n-1}(x)$, hence from Cramer's rule the lemma follows.
\end{proof}


The proof of Theorem \ref{[S2]-THEO-4} easily follows from Lemmas \ref%
{[S3]-LEMA-3}, \ref{[S3]-LEMA-4} and \ref{[S3]-LEMA-5}. Replacing (\ref%
{[S3]-InvR-Pn})--(\ref{[S3]-InvR-Pnm1}) in (\ref{[S3]-DerQ-C1D1}) and (\ref%
{[S3]-DxQnm1-C2D2}) one obtains the ladder equations%
\begin{eqnarray*}
\lbrack Q_{n}^{c,N}(x)]^{\prime } &=&\frac{C_{1}(x;n)B_{2}(x;n)\,\gamma
_{n-1}+D_{1}(x;n)\Lambda _{n-1}^{c}}{\Delta (x;n)\,\gamma _{n-1}}%
Q_{n}^{c,N}(x) \\
&&\quad \quad \quad \quad \quad \quad \quad \quad +\frac{%
D_{1}(x;n)-C_{1}(x;n)\Lambda _{n}^{c}}{\Delta (x;n)}Q_{n-1}^{c,N}(x)
\end{eqnarray*}%
and%
\begin{eqnarray*}
\lbrack Q_{n-1}^{c,N}(x)]^{\prime } &=&\frac{C_{2}(x;n)B_{2}(x;n)\,\gamma
_{n-1}+D_{2}(x;n)\Lambda _{n-1}^{c}}{\Delta (x;n)\,\gamma _{n-1}}%
Q_{n}^{c,N}(x) \\
&&\quad \quad \quad \quad \quad \quad \quad \quad +\frac{%
D_{2}(x;n)-C_{2}(x;n)\Lambda _{n}^{c}}{\Delta (x;n)}Q_{n-1}^{c,N}(x),
\end{eqnarray*}%
which are fully equivalent to (\ref{[S2]-LoweringEq})--(\ref{[S2]-RaisingEq}%
). This completes the proof of Theorem \ref{[S2]-THEO-4}.

Next, the proof of Theorem \ref{[S2]-THEO-5} comes directly from the ladder
operators provided in Theorem \ref{[S2]-THEO-4}. The usual technique (see,
for example \cite[Th. 3.2.3]{Ism05}) consists in applying the raising
operator to both sides of the equation satistied by the lowering operator,
i.e.%
\begin{equation*}
\mathfrak{a}_{n}^{\dag }\left[ \frac{1}{\eta _{1}^{c}(x;n)}\mathfrak{a}%
_{n}[Q_{n}^{c,N}(x)]\right] =\mathfrak{a}_{n}^{\dag }\left[ Q_{n-1}^{c,N}(x)%
\right] ,
\end{equation*}%
which directly implies that%
\begin{equation*}
\mathfrak{a}_{n}^{\dag }\left[ \frac{1}{\eta _{1}^{c}(x;n)}\mathfrak{a}%
_{n}[Q_{n}^{c,N}(x)]\right] =\xi _{2}^{c}(x;n)\,Q_{n}^{c,N}(x)
\end{equation*}%
is a second order differential equation for $Q_{n}^{c,N}(x)$. After some
doable computations, Theorem \ref{[S2]-THEO-5} easily follows.


\section{Zero behavior and electrostatic model for some examples.}

\label{[SECTION-4]-Examples}


Once we have the second order differential equation satisfied by the MOPS $%
\{Q_{n}^{c,N}\}_{n\geq 0}$ it is easy to obtain an electrostatic model for
their zeros (see \cite[Ch. 3]{Ism05}, \cite{Ism00-B}, \cite{HMR-AMC12},
among others). In this Section we shall derive the electrostatic model for
the ceros in case $\mu $ is the Laguerre and the Jacobi classical measures.


\subsection{The Geronimus perturbed Laguerre case}


Let $\{{L_{n}^{\alpha }}${$\}$}$_{n\geq 0}$ be the monic Laguerre
polynomials orthogonal with respect to the Laguerre classical measure $d\mu
_{\alpha }(x)=x^{\alpha }e^{-x}dx$, $\alpha >-1$, supported on $[0,+{\infty }%
)$. We will denote by $\{Q_{n}^{\alpha ,c,N}\}_{n\geq 0}$\ and $%
\{Q_{n}^{\alpha ,c}\}_{n\geq 0}$\ the MOPS corresponding to (\ref%
{[S1]-GeronimusTr}) and (\ref{[S1]-LinearDiv}) when $\mu $ is the Laguerre
classical measure, and $\{y_{n,s}^{\alpha ,c,N}\}_{s=1}^{n}$, $%
\{y_{n,s}^{\alpha ,c}\}_{s=1}^{n}$ their corresponding zeros.

The structure relation (\ref{[S2]-StructRelation})\ for the monic classical
Laguerre polynomials is%
\begin{equation*}
\sigma (x)[L_{n}^{\alpha }(x)]^{\prime }=a(x;n)L_{n}^{\alpha
}(x)+b(x;n)L_{n-1}^{\alpha }(x),
\end{equation*}%
and therefore $\sigma (x)=x$, $a(x;n)=n$, and $b(x;n)=n(n+\alpha )$. Their
three term recurrence relation is%
\begin{equation*}
xL_{n}^{\alpha }(x)=L_{n+1}^{\alpha }(x)+\beta _{n}L_{n}^{\alpha }(x)+\gamma
_{n}L_{n-1}^{\alpha }(x),
\end{equation*}%
with $\beta _{n}=\beta _{n}^{\alpha }=2n+\alpha +1$, $\gamma _{n}=\gamma
_{n}^{\alpha }=n(n+\alpha )$, and the connection formula (\ref%
{[S2]-ConnForm-1}) for $Q_{n}^{\alpha ,c,N}(x)$ in terms of $\{L_{n}^{\alpha
}\}_{n\geq 0}$ reads%
\begin{equation*}
{Q_{n}^{\alpha ,c,N}(x)}=L_{n}^{\alpha }(x)+\Lambda _{n}^{\alpha
,c}\,L_{n-1}^{\alpha }(x).
\end{equation*}%
Taking into account exclusively the coefficients in the above three
expresions, from Theorems \ref{[S2]-THEO-4}\ and \ref{[S2]-THEO-5} we obtain
the explicit expresions for the ladder operators and the coefficients in the
holonomic equation for this first example. After some cumbersome
computations, we get the following set of coefficients (\ref%
{[S3]-Coefs-ABCD1})--(\ref{[S3]-Coefs-ABCD2}) for $\Lambda _{n}^{\alpha
,c}=\Lambda _{n}^{\alpha ,c}(N)$ in (\ref{[S2]-ConnForm-1})%
\begin{eqnarray*}
C_{1}^{\alpha }(x;n) &=&\frac{n-\Lambda _{n}^{\alpha ,c}}{x}, \\
D_{1}^{\alpha }(x;n) &=&\frac{n(n+\alpha )+(x-(n+\alpha ))\Lambda
_{n}^{\alpha ,c}}{x}, \\
A_{2}^{\alpha }(x;n) &=&\frac{-\Lambda _{n}^{\alpha ,c}}{(n-1)(n+\alpha -1)},
\\
B_{2}^{\alpha }(x;n) &=&1+\Lambda _{n-1}^{\alpha ,c}\frac{(x+1-2n+\alpha )}{%
(n-1)(n+\alpha -1)}, \\
C_{2}^{\alpha }(x;n) &=&\frac{-1}{x}-\Lambda _{n-1}^{\alpha ,c}\frac{%
x+1-(n+\alpha )}{x(n-1)(n-1+\alpha )}, \\
D_{2}^{\alpha }(x;n) &=&\frac{x-(n+\alpha )}{x}+\Lambda _{n-1}^{\alpha ,c}%
\frac{(x+1-2n+\alpha )(x-(n+\alpha ))+(x-n(n+\alpha ))}{x(n-1)(n-1+\alpha )}.
\end{eqnarray*}%
Hence, they satisfy the holonomic equation%
\begin{equation}
\lbrack Q_{n}^{\alpha ,c,N}(x)]^{\prime \prime }+\mathcal{R}%
_{L}(x;n)[Q_{n}^{\alpha ,c,N}(x)]^{\prime }+\mathcal{S}_{L}(x;n)Q_{n}^{%
\alpha ,c,N}(x)=0,  \label{[S4]-2ndODELag}
\end{equation}%
with coefficients%
\begin{eqnarray*}
\mathcal{R}_{L}(x;n) &=&-\frac{\Lambda _{n}^{\alpha ,c}}{\Lambda
_{n}^{\alpha ,c}x+\left( n-\Lambda _{n}^{\alpha ,c}\right) \left( n+\alpha
-\Lambda _{n}^{\alpha ,c}\right) }+\frac{\alpha +1}{x}-1, \\
\mathcal{S}_{L}(x;n) &=&\frac{\Lambda _{n}^{\alpha ,c}x+\left( n+\alpha
\right) \left( n-\Lambda _{n}^{\alpha ,c}\right) }{x\left( \Lambda
_{n}^{\alpha ,c}x+\left( n-\Lambda _{n}^{\alpha ,c}\right) \left( n+\alpha
-\Lambda _{n}^{\alpha ,c}\right) \right) }+\frac{n-1}{x}.
\end{eqnarray*}

Now we evaluate (\ref{[S4]-2ndODELag}) at the zeros $\{y_{n,s}^{\alpha
,c,N}\}_{s=1}^{n}$, yielding%
\begin{eqnarray*}
\frac{\lbrack Q_{n}^{\alpha ,c,N}(y_{n,s}^{\alpha ,c,N})]^{\prime \prime }}{%
[Q_{n}^{\alpha ,c,N}(y_{n,s}^{\alpha ,c,N})]^{\prime }} &=&-\mathcal{R}%
_{L}(y_{n,s}^{\alpha ,c,N};n) \\
&=&\frac{\Lambda _{n}^{\alpha ,c}}{\Lambda _{n}^{\alpha ,c}y_{n,s}^{\alpha
,c,N}+\left( n-\Lambda _{n}^{\alpha ,c}\right) \left( n+\alpha -\Lambda
_{n}^{\alpha ,c}\right) }-\frac{\alpha +1}{y_{n,s}^{\alpha ,c,N}}+1.
\end{eqnarray*}%
The above equation reads as the electrostatic equilibrium condition for $%
\{y_{n,s}^{\alpha ,c,N}\}_{s=1}^{n}$. Taking $u_{L}(n;x)=\Lambda
_{n}^{\alpha ,c}x+\left( n-\Lambda _{n}^{\alpha ,c}\right) \left( n+\alpha
-\Lambda _{n}^{\alpha ,c}\right) $, the above condition can be rewritten as
(see \cite{G-JCAM98}, \cite{Ism05}, or \cite{Ism00-B}\ for a treatment of
more general cases and other examples)%
\begin{equation*}
\sum_{j=1,\,j\neq k}^{n}\frac{1}{y_{n,j}^{\alpha ,c,N}-y_{n,k}^{\alpha ,c,N}}%
+\frac{1}{2}\frac{[u_{L}]^{\prime }(n;y_{n,k}^{\alpha ,c,N})}{%
u_{L}(n;y_{n,k}^{\alpha ,c,N})}-\frac{1}{2}\frac{\alpha +1}{y_{n,s}^{\alpha
,c,N}}+\frac{1}{2}=0
\end{equation*}%
which means that the set of zeros $\{y_{n,s}^{\alpha ,c,N}\}_{s=1}^{n}$ are
the critical points of the gradient of the total energy. Hence, the
electrostatic interpretation of the distribution of zeros means that we have
an equilibrium position under the presence of an external potential%
\begin{equation}
V_{L}^{ext}(x)=\frac{1}{2}\text{ln }u_{L}(x;n)-\frac{1}{2}\text{ln }%
x^{\alpha +1}e^{-x},  \label{[S4]-VextLag}
\end{equation}%
where the first term represents a \textit{short range potential}
corresponding to a unit charge located at the unique real zero%
\begin{equation*}
z_{L}(n;x)=\frac{-1}{\Lambda _{n}^{\alpha ,c}}\left( n-\Lambda _{n}^{\alpha
,c}\right) \left( n+\alpha -\Lambda _{n}^{\alpha ,c}\right) 
\end{equation*}%
of the linear polynomial $u_{L}(x;n)$, and the second one is a \textit{long
range potential\ }associated with the Laguerre weight function.

Finally, in order to illustrate the results of Theorem \ref{[S2]-THEO-2}, we
consider the Geronimus perturbation on the Laguerre measure with $\alpha =0$
and $c=-1$%
\begin{equation}
d\nu _{N}(x)=\frac{1}{(x+1)}e^{-x}dx+N\delta (x+1),\quad N\geq 0,
\label{[S4]-Ger-Lag-measure}
\end{equation}%
and we obtain the behavior of the zeros $\{y_{n,s}^{0,-1,N}\}_{s=1}^{n}$ as $%
N$ increases. We enclose in Figure \ref{[S4]-FigLag} the graphs of ${%
L_{3}^{0}(x)}$ (dotted line), $Q_{3}^{0,-1}(x)$ (dash-dotted line), and $%
Q_{3}^{0,-1,N}(x)$ for some $N$, to show the monotonicity of their zeros as
a function of the mass $N$. 
\begin{figure}[th]
\centerline{\includegraphics[width=11cm,keepaspectratio]{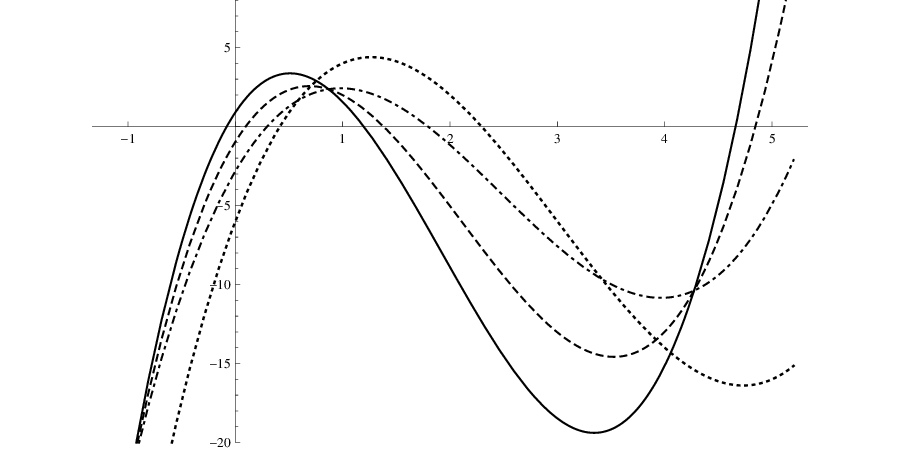}}
\caption{The graphs of $L_{3}^{0}(x)$ (dotted) and $Q_{3}^{0,-1,N}(x)$ for
some values of $N$.}
\label{[S4]-FigLag}
\end{figure}
Table \ref{[S4]-TabLag} shows the behavior of the zeros of $%
Q_{3}^{0,-1,N}(x) $ for several choices of $N$. Observe that the smallest
zero converges to $c=-1$ and the other two zeros converge to the zeros of
the monic kernel polynomial $L_{2}^{0 ,-1,[1]}(x)$, as is shown in Theorem %
\ref{[S2]-THEO-2}. That is, they converge to $x_{2,1}^{0,-1,[1]}=0.869089$
and $x_{2,2}^{0,-1,[1]}=4.273768$. Notice that all the zeros decrease as $N$
increases. The zeros outside the interval $[0,+\infty )$, namely the support
of the classical Laguerre measure, appear in bold.

\begin{table}[th]
\centering{\small \renewcommand{\arraystretch}{1.5}%
\begin{tabular}{|ccccccccc|}
\hline
\emph{N} &  & \emph{1st} &  & \emph{2nd} &  & \emph{3rd} &  & \emph{z(N)} \\ 
\hline
\multicolumn{1}{|r}{$0$} & \multicolumn{1}{r}{} & \multicolumn{1}{r}{$%
0.296771$} & \multicolumn{1}{r}{} & \multicolumn{1}{r}{$1.794881$} & 
\multicolumn{1}{r}{} & \multicolumn{1}{r}{$5.327153$} & \multicolumn{1}{r}{}
& \multicolumn{1}{r|}{$-1.27309$} \\ 
\multicolumn{1}{|r}{$0.0125$} & \multicolumn{1}{r}{} & \multicolumn{1}{r}{$%
0.096936$} & \multicolumn{1}{r}{} & \multicolumn{1}{r}{$1.381317$} & 
\multicolumn{1}{r}{} & \multicolumn{1}{r}{$4.846199$} & \multicolumn{1}{r}{}
& \multicolumn{1}{r|}{$-0.039345$} \\ 
\multicolumn{1}{|r}{$0.025$} & \multicolumn{1}{r}{} & \multicolumn{1}{r}{$%
\mathbf{-0.079531}$} & \multicolumn{1}{r}{} & \multicolumn{1}{r}{$1.196907$}
& \multicolumn{1}{r}{} & \multicolumn{1}{r}{$4.66079$} & \multicolumn{1}{r}{}
& \multicolumn{1}{r|}{$-0.015274$} \\ 
\multicolumn{1}{|r}{$0.05$} & \multicolumn{1}{r}{} & \multicolumn{1}{r}{$%
\mathbf{-0.324373}$} & \multicolumn{1}{r}{} & \multicolumn{1}{r}{$1.050055$}
& \multicolumn{1}{r}{} & \multicolumn{1}{r}{$4.50679$} & \multicolumn{1}{r}{}
& \multicolumn{1}{r|}{$-0.156362$} \\ 
\multicolumn{1}{|r}{$5$} & \multicolumn{1}{r}{} & \multicolumn{1}{r}{$%
\mathbf{-0.988481}$} & \multicolumn{1}{r}{} & \multicolumn{1}{r}{$0.87094$}
& \multicolumn{1}{r}{} & \multicolumn{1}{r}{$4.276644$} & \multicolumn{1}{r}{
} & \multicolumn{1}{r|}{$-0.700057$} \\ \hline
\end{tabular}%
}
\caption{Zeros of $Q_{3}^{0,-1,N}(x)$ and $z(0,-1,3,N;x)$ for some values of 
$N$.}
\label{[S4]-TabLag}
\end{table}


\begin{remark}
Looking at the external potential (\ref{[S4]-VextLag}) there are few
significant differences with respect to the Uvarov case (see \cite[Sec. 4.2]%
{HMR-AMC12}). First, the long range potential does not depend on the shift $%
c $, as occurs in the Uvarov case, where the long range potential
corresponds to a polynomial perturbation of the Laguerre measure by $(x-c)$.
Second, in the Uvarov case the polynomial $u_{L}(x;n)$ has two different
real roots when $c<0$, away from the boundary $[0,+\infty )$, meanwhile in
this case there exists only one real root for $u_{L}(x;n)$. It means that
the inclusion of the linear rational modification of the measure present in
the Geronimus transformation has notable dynamical consequences on the
electrostatic model.
\end{remark}



\subsection{The Geronimus perturbed Jacobi case}


The electrostatic model in case $\mu $\ is the Jacobi classical measure is
essentially the same as in the former case, but considering the
corresponding values and expressions for the Jacobi measure, and the shift $%
c $ away of its support. Since the exact formulas are cumbersome, we will
not write them all down, except the most significant ones.

Let $\{{P_{n}^{\alpha ,\beta }}${$\}$}$_{n\geq 0}$ be the monic Jacobi
polynomials orthogonal with respect to the Jacobi classical measure $d\mu
_{\alpha ,\beta }(x)=(1-x)^{\alpha }(1+x)^{\beta }dx$, $\alpha ,\beta >-1$,
supported on $[-1,1]$. We will denote by $\{Q_{n}^{\alpha
,\beta,c,N}\}_{n\geq 0}$\ and $\{Q_{n}^{\alpha ,\beta ,c}\}_{n\geq 0}$\ the
MOPS corresponding to (\ref{[S1]-GeronimusTr}) and (\ref{[S1]-LinearDiv})
when $\mu $ is the Jacobi measure, and $\{y_{n,s}^{\alpha ,\beta
,c,N}\}_{s=1}^{n}$, $\{y_{n,s}^{\alpha ,\beta ,c}\}_{s=1}^{n}$ their
corresponding zeros.

The structure relation for the monic Jacobi polynomials reads%
\begin{equation*}
\sigma (x)[P_{n}^{\alpha ,\beta }(x)]^{\prime }=a(x;n)P_{n}^{\alpha ,\beta
}(x)+b(x;n)P_{n-1}^{\alpha ,\beta }(x),
\end{equation*}%
with%
\begin{eqnarray*}
\sigma (x) &=&(1-x^{2}), \\
a(x;n) &=&-n(1+x)+\frac{2n(n+\alpha )}{(2n+\alpha +\beta )}, \\
b(x;n) &=&\frac{4n(n+\alpha )(n+\beta )(n+\alpha +\beta )}{(2n+\alpha +\beta
+1)(2n+\alpha +\beta )^{2}}.
\end{eqnarray*}%
Their three term recurrence relation is%
\begin{equation*}
xP_{n}^{\alpha ,\beta }(x)=P_{n+1}^{\alpha ,\beta }(x)+\beta
_{n}P_{n}^{\alpha ,\beta }(x)+\gamma _{n}P_{n-1}^{\alpha ,\beta }(x),
\end{equation*}%
with%
\begin{eqnarray*}
\beta _{n} &=&\beta _{n}^{\alpha ,\beta }=\frac{\beta ^{2}-\alpha ^{2}}{%
(2n+\alpha +\beta )(2n+\alpha +\beta +2)}, \\
\gamma _{n} &=&\gamma _{n}^{\alpha ,\beta }=\frac{4n(n+\alpha )(n+\beta
)(n+\alpha +\beta )}{(2n+\alpha +\beta -1)(2n+\alpha +\beta )^{2}(2n+\alpha
+\beta +1)},
\end{eqnarray*}%
and the connection formula (\ref{[S2]-ConnForm-1}) for $Q_{n}^{\alpha
,c,N}(x)$ in terms of $\{P_{n}^{\alpha ,\beta }\}_{n\geq 0}$ is%
\begin{equation*}
{Q_{n}^{\alpha ,\beta ,c,N}(x)}=P_{n}^{\alpha ,\beta }(x)+\Lambda
_{n}^{\alpha ,\beta ,c}\,P_{n-1}^{\alpha ,\beta }(x).
\end{equation*}%
The coefficient of $[Q_{n}^{\alpha ,\beta ,c,N}(x)]^{\prime }$ in the
holonomic equation is%
\begin{equation*}
\mathcal{R}_{L}(x;n)=-\frac{[u_{J}]^{\prime }(n;x)}{u_{J}(n;x)}-\frac{%
2x-\beta \left( 1-x\right) +\alpha \left( 1+x\right) }{(1-x)(1+x)},
\end{equation*}%
with%
\begin{eqnarray}
u_{J}(n;x) &=&4n(n+\alpha )(n+\beta )(n+\alpha +\beta )+(2n+\alpha +\beta
-1)(2n+\alpha +\beta )\Lambda _{n}^{\alpha ,\beta ,c}\cdot 
\label{[S4]-unJac} \\
&&\left[ \left( 2n+\alpha +\beta \right) ^{2}x+\left( \alpha +\beta \right)
\left( \alpha -\beta \right) +(2n+\alpha +\beta -1)(2n+\alpha +\beta
)\Lambda _{n}^{\alpha ,\beta ,c}\right] .  \notag
\end{eqnarray}%
Hence, the electrostatic equilibrium means that the set of zeros $%
\{y_{n,s}^{\alpha ,\beta ,c,N}\}_{s=1}^{n}$ have an equilibrium position
under the presence of the external potential%
\begin{equation*}
V_{J}^{ext}(x)=\frac{1}{2}\text{ln }u_{J}(x;n)-\frac{1}{2}\text{ln }%
(1-x)^{\alpha +1}(1+x)^{\beta +1},
\end{equation*}%
where the first term represents a \textit{short range potential}
corresponding to a unit charge located at the real root%
\begin{eqnarray*}
z_{J}(x;n) &=&-\frac{(\alpha ^{2}-\beta ^{2})(2n+\alpha +\beta )+\frac{%
4n(n+\alpha )(n+\beta )(n+\alpha +\beta )}{(2n+\alpha +\beta -1)\Lambda
_{n}^{\alpha ,\beta ,c}}}{(2n+\alpha +\beta )^{3}} \\
&&\qquad -\frac{(2n+\alpha +\beta -1)(2n+\alpha +\beta )^{2}\Lambda
_{n}^{\alpha ,\beta ,c}}{(2n+\alpha +\beta )^{3}}
\end{eqnarray*}%
of (\ref{[S4]-unJac}), and the second one is a \textit{long range potential\ 
}associated with the Jacobi weight function. Observe that, as in the
Laguerre case, the long range potential does not depend on the shift $c$.

Finally, we analyze the consequences of Theorem \ref{[S2]-THEO-2} and \ref%
{[S2]-THEO-3}), for a Geronimus perturbation on the Laguerre measure with $%
\alpha =0.5$, $\beta =1$ and $c=-1.5$%
\begin{equation*}
d\nu _{N}(x)=\frac{(1-x)^{0.5}(1+x)}{(x+1.5)}dx+N\delta (x+1.5),\quad N\geq
0,
\end{equation*}%
and we obtain the behavior of the zeros $\{y_{n,s}^{0.5,1,-1.5,N}%
\}_{s=1}^{n} $ as $N$ increases. We provide in Figure \ref{[S4]-FigLag} the
graphs of ${P_{4}^{0.5,1}(x)}$ (dotted line), $Q_{4}^{0.5,1,-1}(x)$
(dash-dotted line), and $Q_{4}^{0.5,1,-1.5,N}(x)$ for some $N$, to show the
monotonicity of their zeros as a function of the mass $N$ (See Figure \ref%
{[S4]-FigJac}). 
\begin{figure}[th]
\centerline{\includegraphics[width=11cm,keepaspectratio]{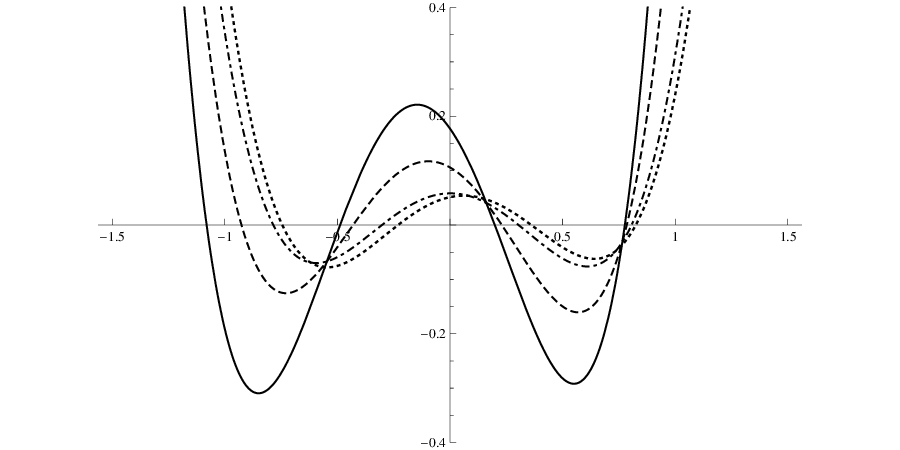}}
\caption{The graphs of $P_{4}^{0.5,1}(x)$ (dotted) and $%
Q_{4}^{0.5,1,-1.5,N}(x)$ for some values of $N$.}
\label{[S4]-FigJac}
\end{figure}
Table \ref{[S4]-TabJac} shows the zeros $\{y_{4,s}^{\alpha ,\beta
,c,N}\}_{s=1}^{4}$ of $Q_{4}^{\alpha ,\beta ,c,N}(x)$ having $\alpha =0.5$, $%
\beta =1$, and $c=-1.5$ for several values of $N$. Observe that the smallest
zero $y_{4,1}^{0.5,1,-1.5,N}$ converges to $c=-1.5$ and the other three
zeros converge to the zeros $\{x_{3,s}^{0.5,1,-1.5,[1]}\}_{s=1}^{3}$ of the
kernel polynomial $P_{3}^{0.5,1,-1.5,[1]}(x)$, as states Theorem \ref%
{[S2]-THEO-2}. That is, they converge respectively to $%
x_{3,1}^{0.5,1,-1.5,[1]}=-0.546629$, $x_{3,2}^{0.5,1,-1.5,[1]}=0.161665$,
and $x_{3,3}^{0.5,1,-1.5,[1]}=0.765232$. The zeros outside the interval $%
[-1,1]$, namely the support of the classical Jacobi measure, appear in bold.

\begin{table}[th]
\centering{\small \renewcommand{\arraystretch}{1.5}%
\begin{tabular}{|ccccccccccc|}
\hline
\emph{N} &  & \emph{1st} &  & \emph{2nd} &  & \emph{3rd} &  & \emph{4rd} & 
& \emph{z(N)} \\ \hline
\multicolumn{1}{|r}{$0$} & \multicolumn{1}{r}{} & \multicolumn{1}{r}{$%
-0.784545$} & \multicolumn{1}{r}{} & \multicolumn{1}{r}{$-0.302212$} & 
\multicolumn{1}{r}{} & \multicolumn{1}{r}{$0.304654$} & \multicolumn{1}{r}{}
& \multicolumn{1}{r}{$0.806277$} & \multicolumn{1}{r}{} & 
\multicolumn{1}{r|}{$-1.61637$} \\ 
\multicolumn{1}{|r}{$0.0008$} & \multicolumn{1}{r}{} & \multicolumn{1}{r}{$%
-0.925906$} & \multicolumn{1}{r}{} & \multicolumn{1}{r}{$-0.430453$} & 
\multicolumn{1}{r}{} & \multicolumn{1}{r}{$0.230271$} & \multicolumn{1}{r}{}
& \multicolumn{1}{r}{$0.784909$} & \multicolumn{1}{r}{} & 
\multicolumn{1}{r|}{$-0.97778$} \\ 
\multicolumn{1}{|r}{$0.0020$} & \multicolumn{1}{r}{} & \multicolumn{1}{r}{$%
\mathbf{-1.080633}$} & \multicolumn{1}{r}{} & \multicolumn{1}{r}{$-0.488136$}
& \multicolumn{1}{r}{} & \multicolumn{1}{r}{$0.199190$} & \multicolumn{1}{r}{
} & \multicolumn{1}{r}{$0.776221$} & \multicolumn{1}{r}{} & 
\multicolumn{1}{r|}{$-1.04893$} \\ 
\multicolumn{1}{|r}{$0.05$} & \multicolumn{1}{r}{} & \multicolumn{1}{r}{$%
\mathbf{-1.467364}$} & \multicolumn{1}{r}{} & \multicolumn{1}{r}{$-0.544057$}
& \multicolumn{1}{r}{} & \multicolumn{1}{r}{$0.163585$} & \multicolumn{1}{r}{
} & \multicolumn{1}{r}{$0.765818$} & \multicolumn{1}{r}{} & 
\multicolumn{1}{r|}{$-1.35837$} \\ 
\multicolumn{1}{|r}{$5$} & \multicolumn{1}{r}{} & \multicolumn{1}{r}{$%
\mathbf{-1.499661}$} & \multicolumn{1}{r}{} & \multicolumn{1}{r}{$-0.546604$}
& \multicolumn{1}{r}{} & \multicolumn{1}{r}{$0.161684$} & \multicolumn{1}{r}{
} & \multicolumn{1}{r}{$0.765238$} & \multicolumn{1}{r}{} & 
\multicolumn{1}{r|}{$-1.38587$} \\ \hline
\end{tabular}%
}
\caption{Zeros of $Q_{4}^{0.5,1,-1.5,N}(x)$ and $z_{J}(x;n)$ for some values
of $N$.}
\label{[S4]-TabJac}
\end{table}


\section{Appendix. The interlacing lemma}

\label{[SECTION-5]-Appendix}


Next, we will analyze the behavior of zeros of polynomial of the form $%
f(x)=h_{n}(x)+cg_{n}(x)$. We need the following lemma concerning the
behavior and the asymptotics of the zeros of linear combinations of two
polynomials with interlacing zeros (see \cite[Lemma 1]{BDR-JCAM02}, \cite[%
Lemma 3]{DMR-ANM10} for a detailed discussion).

\begin{lemma}
\label{InterlacingLemma} Let $h_{n}(x)=a(x-x_{1})\cdots (x-x_{n})$ and $%
g_{n}(x)=b(x-\zeta _{1})\cdots (x-\zeta _{n})$ be polynomials with real and
simple zeros, where $a$ and $b$ are real positive constants.

\begin{itemize}
\item[$(i)$] If%
\begin{equation*}
\zeta _{1}<x_{1}<\cdots <\zeta _{n}<x_{n},
\end{equation*}%
then, for any real constant $c>0$, the polynomial 
\begin{equation*}
f(x)=h_{n}(x)+cg_{n}(x)
\end{equation*}%
has $n$ real zeros $\eta _{1}<\cdots <\eta _{n}$ which interlace with the
zeros of $h_{n}(x)$ and $g_{n}(x)$ in the following way 
\begin{equation*}
\zeta _{1}<\eta _{1}<x_{1}<\cdots <\zeta _{n}<\eta _{n}<x_{n}.
\end{equation*}%
Moreover, each $\eta _{k}=\eta _{k}(c)$ is a decreasing function of $c$ and,
for each $k=1,\ldots ,n$, 
\begin{equation*}
\lim_{c\rightarrow \infty }\eta _{k}=\zeta _{k}\quad \text{and}\quad
\lim_{c\rightarrow \infty }c[\eta _{k}-\zeta _{k}]=\dfrac{-h_{n}(\zeta _{k})%
}{g_{n}^{\prime }(\zeta _{k})}.
\end{equation*}

\item[$(ii)$] If%
\begin{equation*}
x_{1}<\zeta _{1}<\cdots <x_{n}<\zeta _{n},
\end{equation*}%
then, for any positive real constant $c>0$, the polynomial 
\begin{equation*}
f(x)=h_{n}(x)+cg_{n}(x)
\end{equation*}%
has $n$ real zeros $\eta _{1}<\cdots <\eta _{n}$ which interlace with the
zeros of $h_{n}(x)$ and $g_{n}(x)$ as follows 
\begin{equation*}
x_{1}<\eta _{1}<\zeta _{1}<\cdots <x_{n}<\eta _{n}<\zeta _{n}.
\end{equation*}%
Moreover, each $\eta _{k}=\eta _{k}(c)$ is an increasing function of $c$
and, for each $k=1,\ldots ,n$, 
\begin{equation*}
\lim_{c\rightarrow \infty }\eta _{k}=\zeta _{k}\quad \text{and}\quad
\lim_{c\rightarrow \infty }c[\zeta _{k}-\eta _{k}]=\dfrac{h_{n}(\zeta _{k})}{%
g_{n}^{\prime }(\zeta _{k})}.
\end{equation*}
\end{itemize}
\end{lemma}



\begin{thebibliography}{99}
\bibitem{BB-JDEA09} D. Barrios and A. Branquinho, \emph{Complex high order
Toda and Volterra lattices}, J. Difference Equ. Appl. \textbf{15} (2)
(2009), 197--213.

\bibitem{BDR-JCAM02} C. F. Bracciali, D. K. Dimitrov, and A. Sri Ranga, 
\emph{Chain sequences and symmetric generalized orthogonal polynomials}, J.
Comput. Appl. Math. \textbf{143} (2002), 95--106.

\bibitem{BM-IJMMS96} A. Branquinho and F. Marcell\'an, \emph{Generating new
classes of orthogonal polynomials}, Int. J. Math. Math. Sci. \textbf{19} (4)
(1996), 643--656.

\bibitem{BDT-NA10} M. I. Bueno, A. Deaño, and E. Tavernetti, \emph{A new
algorithm for computing the Geronimus transformation with large shifts},
Numer. Alg. \textbf{54} (2010), 101--139.

\bibitem{BM-LAA04} M. I. Bueno and F. Marcellán, \emph{Darboux
tranformations and perturbation of linear functionals}, Linear Algebra Appl. 
\textbf{384} (2004), 215--242.

\bibitem{Chi78} T. S. Chihara, \emph{An Introduction to Orthogonal
Polynomials}. Mathematics and its Applications Series, Gordon and Breach,
New York, 1978.

\bibitem{DM-NA13} M. Derevyagin and F. Marcellán, \emph{A note on the
Geronimus transformation and Sobolev orthogonal polynomials}, Numer.
Algorithms, DOI 10.1007/s11075-013-9788-6, (2013), 1-17.

\bibitem{DMR-ANM10} D. K. Dimitrov, M. V. Mello, and F. R. Rafaeli, \emph{%
Monotonicity of zeros of Jacobi-Sobolev type orthogonal polynomials}, Appl.
Numer. Math. \textbf{60} (2010), 263--276.

\bibitem{DM-ANM90} J. Dini, P. Maroni, \emph{La multiplication d'une forme
linéaire par une forme rationnelle. Application aux polynômes de
Laguerre-Hahn}, Ann. Polon. Math. \textbf{52} (1990), 175--185.

\bibitem{Gauts04} W. Gautschi, \emph{Orthogonal Polynomials: Computation and
Approximation}, in Numerical Mathematics and Scientific Computation Series,
Oxford University Press. New York. 2004.

\bibitem{G-HNeds40} Y. L. Geronimus, \emph{On the polynomials orthogonal
with respect to a given number sequence and a theorem}, in: W. Hahn and I.
A. Nauk, (eds.), vol. \textbf{4} (1940), 215--228 (in Russian).

\bibitem{G-ZMOK40} Y. L. Geronimus, \emph{On the polynomials orthogonal with
respect to a given number sequence}, Zap. Mat. Otdel. Khar'kov. Univers. i
NII Mat. i Mehan. \textbf{17} (1940), 3--18.

\bibitem{G-JCAM98} F. A. Grünbaum, \emph{Variations on a theme of Heine and
Stieltjes: An electrotatic interpretation of the zeros of certain polynomials%
}, J. Comp. Appl. Math. \textbf{99} (1998), 189--194.

\bibitem{HMR-AMC12} E. J. Huertas, F. Marcellán and F. R. Rafaeli, \emph{%
Zeros of orthogonal polynomials generated by canonical perturbations of
measures}, Appl. Math. Comput. \textbf{218} (13) (2012), 7109-7127.

\bibitem{HR-JCAM85} E. Hendriksen and H. Van Rossum, \emph{Semiclassical
orthogonal polynomials}, Lect. Notes in Math., Springer Verlag (1985) Vol. 
\textbf{1171}, Cl. Brezinski et coll. ed., p. 354--361.

\bibitem{Ism00-B} M. E. H. Ismail, \emph{More on electrostatic models for
zeros of orthogonal polynomials}, Numer. Funct. Anal. Optimiz. \textbf{21}
(2000), 191--204.

\bibitem{Ism05} M. E. H. Ismail, \emph{Classical and Quantum Orthogonal
Polynomials in One Variable}, Encyclopedia of Mathematics and its
Applications, Vol. \textbf{98}. Cambridge University Press, Cambridge UK.
2005.

\bibitem{MarcMar92} F. Marcellán and P. Maroni, \emph{Sur l'adjonction d'
une masse de Dirac à une forme réguliére et semi-classique}, Annal. Mat.
Pura ed Appl. \textbf{162} (1992), 1--22.

\bibitem{M-PMH90} P. Maroni, \emph{Sur la suite de polynômes orthogonaux
associée à la forme }$u=\delta _{c}+\lambda (x-c)^{-1}L$, Period. Math.
Hungar. \textbf{21} (3) (1990), 223--248.

\bibitem{Mar91} P. Maroni, \emph{Une théorie algébrique des polynômes
orthogonaux. Application aux polynômes orthogonaux semi-classiques}, in\emph{%
Orthogonal Polynomials and Their Applications}, C. Brezinski et al. Editors.
Annals. Comput. Appl. Math. \textbf{9}. Baltzer, Basel. 1991, 95--130.

\bibitem{RM-CMB89} A. Ronveaux and F. Marcellán, \emph{Differential equation
for classical-type orthogonal polynomials}, Canad. Math. Bull. \textbf{32}
(4), (1989), 404--411.

\bibitem{S-TAMS37} J. Shohat, \emph{On mechanical quadratures, in
particular, with positive coefficients}, Trans. Amer. Math. Soc. \textbf{42}
(3) (1937), 461--496.

\bibitem{SZ-MAA95} V. Spiridonov, A. Zhedanov, \emph{Discrete Darboux
transformations, the discrete-time Toda lattice, and the Askey-Wilson
polynomials}, Methods Appl. Anal. \textbf{2} (4) (1995), 369--398.

\bibitem{SZ-JPA97} V. Spiridonov, A. Zhedanov, \emph{Discrete-time Volterra
chain and classical orthogonal polynomials}, J. Phys. A: Math. Gen. \textbf{%
30} (1997), 8727--8737.

\bibitem{Sze75} G. Szeg\H{o}, \emph{Orthogonal Polynomials}, Amer. Math.
Soc. Coll. Publ, Vol. \textbf{23}, 4th ed., Amer. Math. Soc., Providence,
RI, 1975.

\bibitem{U-UCMP69} V. B. Uvarov, \emph{The connection between systems of
polynomials orthogonal with respect to different distribution functions},
USSR Compt. Math. Phys. \textbf{9} (6) (1969), 25--36.

\bibitem{Z-JCAM97} A. Zhedanov, \emph{Rational spectral transformations and
orthogonal polynomials}, J. Comput. Appl. Math. \textbf{85} (1997), 67--83.

\bibitem{Y-BKMS02} G. J. Yoon,\emph{Darboux transforms and orthogonal
polynomials}, Bull. Korean Math. Soc. \textbf{39} (2002), 359--376.
\end{thebibliography}
\end{document}